\documentclass[a4paper]{article}

\usepackage[latin1]{inputenc}
\usepackage[T1]{fontenc}
\usepackage{amssymb,amsmath,amsthm,amscd,mathrsfs}
\usepackage[all]{xy}
\usepackage{color}
\usepackage[margin=3cm]{geometry}
\usepackage{amsmath}
\usepackage{paralist}

\newcommand{\bC}{{\mathbb C}}

\newcommand{\bN}{{\mathbb N}}

\newcommand{\bP}{{\mathbb P}}

\newcommand{\bR}{{\mathbb R}}

\newcommand{\bZ}{{\mathbb Z}}

\newcommand{\cP}{{\mathscr P}}

\newcommand{\cW}{{\mathscr W}}
\newcommand{\cX}{{\mathscr X}}
\newcommand{\cY}{{\mathscr Y}}

\newcommand{\dC}{{\mathcal C}}
\newcommand{\dD}{{\mathcal D}}

\newcommand{\dG}{{\mathcal G}}

\newcommand{\dK}{{\mathcal K}}

\newcommand{\dM}{{\mathcal M}}

\renewcommand{\phi}{\varphi}

\DeclareMathOperator{\iso}{\cong}

\DeclareMathOperator{\pr}{pr}
\DeclareMathOperator{\inj}{\hookrightarrow}

\DeclareMathOperator{\NS}{NS}
\DeclareMathOperator{\del}{\partial}

\DeclareMathOperator{\Mon}{Mon}
\DeclareMathOperator{\MonHdg}{Mon^2_\text{Hdg}}

\DeclareMathOperator{\Def}{Def}

\renewcommand{\div}{{\rm div}}

\DeclareMathOperator{\Amp}{Amp}

\DeclareMathOperator{\Mov}{Mov}

\newcommand{\TODO}[1]{}

\newtheorem{thm}{Theorem}[section]
\newtheorem{defi}[thm]{Definition}
\newtheorem{prop}[thm]{Proposition}
\newtheorem{lemme}[thm]{Lemma}
\newtheorem{cor}[thm]{Corollary}

\theoremstyle{remark}
\newtheorem{remark}[thm]{Remark}
\newtheorem{nota}[thm]{Notation}
\newtheorem{rmk}[thm]{Remark}
\newtheorem{ex}[thm]{Example}

\DeclareMathOperator{\Z}{\mathbb{Z}}
\DeclareMathOperator{\R}{\mathbb{R}}

\DeclareMathOperator{\K}{\mathbb{K}}

\DeclareMathOperator{\codim}{Codim}

\DeclareMathOperator{\Pic}{Pic}
\DeclareMathOperator{\rk}{rk}
\DeclareMathOperator{\Gr}{Gr}

\DeclareMathOperator{\Image}{Im}
\DeclareMathOperator{\Ima}{Im}
\DeclareMathOperator{\Rea}{Re}

\DeclareMathOperator{\Vect}{Vect}

\DeclareMathOperator{\id}{id}

\DeclareMathOperator{\Sing}{Sing}

\DeclareMathOperator{\GL}{GL}
\DeclareMathOperator{\Q}{\mathbb{Q}}

\DeclareMathOperator{\C}{\mathbb{C}}

\DeclareMathOperator{\BK}{\mathcal{B}\mathcal{K}}
\DeclareMathOperator{\coloneqq}{:=}

\DeclareMathOperator{\Ree}{Re}
\DeclareMathOperator{\Dom}{Dom}

\DeclareMathOperator{\Ker}{Ker}

\newcommand{\eq}[1][r]
{\ar@<-3pt>@{-}[#1]
\ar@<-1pt>@{}[#1]|<{}="gauche"
\ar@<+0pt>@{}[#1]|-{}="milieu"
\ar@<+1pt>@{}[#1]|>{}="droite"
\ar@/^2pt/@{-}"gauche";"milieu"
\ar@/_2pt/@{-}"milieu";"droite"}

\newcommand{\incl}[1][r]
  {\ar@<-0.2pc>@{^(-}[#1] \ar@<+0.2pc>@{-}[#1]}

\makeatletter
\newcommand{\doublewidetilde}[1]{{%
  \mathpalette\double@widetilde{#1}%
}}
\newcommand{\double@widetilde}[2]{%
  \sbox\z@{$\m@th#1\widetilde{#2}$}%
  \ht\z@=.9\ht\z@
  \widetilde{\box\z@}%
}
\makeatother

\date{\today\thanks{Apart Section \ref{Kawa} and some details in Section \ref{sec:wall-divs}, this paper has been written before September 2020. Section \ref{Kawa} was a request of the previous referee and has been written in September 2023}}  
\begin{document}
\title{\bf On the Kähler cone of irreducible symplectic orbifolds}

\author{Gr\'egoire \textsc{Menet}; Ulrike \textsc{Riess}}

\maketitle


\begin{abstract} 
We generalize the Hodge version of the global Torelli theorem in the framework of irreducible symplectic orbifolds.
We also propose a generalization of several results related to the Kähler cone and the notion of wall
divisors introduced in the smooth case by Mongardi in \cite{Mongardi13}. As an application we propose a
definition of the mirror symmetry as an involution on a moduli space; this construction is also new in the smooth case.
\end{abstract}

\section{Introduction}
A compact Kähler analytic space is called \emph{holomorphically symplectic} if it admits a holomorphic 2-form
which is non-degenerate on its smooth locus. In the last years there have been many developments regarding
these objects with the objective of generalizing parts of the existing theory for compact hyperkähler
manifolds. These generalizations respond to two necessities. The first one is to get around the lack of examples
in the smooth case and the second is to find a framework more adapted to the minimal model program, where
certain singularities naturally arise. The two most important developments in this area during the last years
concern the Beauville--Bogomolov decomposition theorem \cite{Greb}, \cite{Horing-Peternell} and the global Torelli theorem \cite{Bakker-Lehn-GlobalTorelli}, \cite{Menet-2020}.
This paper continues the quest in the framework of irreducible symplectic orbifolds.

A complex analytic space is called an \emph{orbifold} if it is locally isomorphic to a quotient of an open
subset of $\C^n$ by a finite automorphism group (Definition \ref{VV}). An orbifold $X$ is called
\emph{irreducible (holomorphically) symplectic} if $X\smallsetminus \Sing X$ is simply connected, admits a
unique (up to a scalar multiple), nondegenerate holomorphic 2-form, and $\codim \Sing X\geq 4$ (Definition
\ref{def}).  Irreducible symplectic orbifolds have several key properties, making them a particularly
interesting class of geometrical objects. First, they appear as elementary bricks in a Bogomolov
decomposition theorem (\cite[Theorem 6.4]{Campana-2004}). Second, they are well adapted to the minimal model
program since they are $\Q$-factorial and have terminal singularities (\cite[Corollary 1]{Namikawa:note}). Third, the
construction of twistor spaces (\cite[Theorem 5.4]{Menet-2020}) generalizes to these objects which is a
powerful tool when generalizing results from the smooth case. Finally, many examples have already been constructed in this framework (\cite[Section 13]{Fujiki-1983}).

After the generalization of the global Torelli theorem in the moduli space setting (\cite[Theorem 1.1]{Menet-2020}), it seems natural to generalize its Hodge version (see \cite[Theorem 1.3]{Markman11} in the smooth case).
\begin{thm}\label{main}
Let $X$ and $X'$ be irreducible symplectic orbifolds.
\begin{itemize}
\item[(i)]
$X$ and $X'$ are bimeromorphic if and only if there exists a parallel transport operator $f:H^2(X,\Z)\rightarrow H^2(X',\Z)$ which is an isometry of integral Hodge structures.
\item[(ii)]
Let $f:H^2(X,\Z)\rightarrow H^2(X',\Z)$ be a parallel transport operator, which is an isometry of integral Hodge structures. There exists an isomorphism $\widetilde{f}:X\rightarrow X'$ such that $f=\widetilde{f}_*$ if and only if $f$ maps some K\"ahler class on $X$ to a K\"ahler class on $X'$.
\end{itemize}
\end{thm}
As in the smooth case, the second cohomology group of an irreducible symplectic orbifold is endowed with the Beauville--Bogomolov form (see \cite[Theorem 3.17]{Menet-2020}).  
See Definition \ref{transp} for the definition of parallel transport operator.

The proof of Theorem \ref{main} uses a generalization of Huybrechts' theorem on bimeromorphic hyperkähler manifolds \cite[Theorem 2.5]{Huybrechts2003} to the orbifold setting (see Proposition \ref{mainprop2}).

Theorem \ref{main} shows the importance of understanding the Kähler classes on irreducible symplectic orbifolds.

A central focus of this paper is the development of the theory of wall
divisors for irreducible symplectic orbifolds,
thus generalizing the theory independently developed by 
Mongardi \cite{Mongardi13} and Amerik--Verbitsky \cite{AmerikVerbitsky14}. By definition, the wall divisors provide orthogonal hyperplanes which (even after applying
Hodge preserving monodromy operators) intersect trivially the
birational Kähler chambers (see Definition \ref{walldefi}). We prove that wall divisors are invariant
under arbitrary parallel transport and that extremal rays of the cone of effective curves give rise to wall divisors.
More precisely, we have the following result:
\begin{thm}\label{thm:defo-for-WDintro}
  \begin{itemize}
    \item[(i)] (compare Theorem \ref{thm:defo-for-WD}) Let $X$ and $Y$ be two irreducible symplectic orbifolds such that there exists a parallel transport operator $f: H^2(X,\Z)\rightarrow H^2(Y,\Z)$.
      Let $D$ be a wall divisor of $X$ such that $f(D)\in \Pic(Y)$. Then $f(D)$ is a wall divisor on $Y$.
    \item[(ii)] (compare Proposition \ref{extremalray}) Let $R$ be an extremal ray of the cone of effective curves on $X$ with negative self intersection
      with respect to the Beauville--Bogomolov form. 
			Then the hyperplane $R^\perp \subseteq H^2(X,\bZ)$ is the
      orthogonal complement of a wall divisor.
  \end{itemize}
\end{thm}
The proof of this theorem adopts a different approach from \cite[Proposition 7.4]{Lehn2}, as it is based on the existence of the twistor space for irreducible symplectic orbifolds. This strategy allows us to dispense with both the projectivity assumption on $X$ and the condition $b_2(X)\geq 5$. 

These results have many applications. For example, they give information on the singularities of the irreducible symplectic orbifolds (see Section \ref{applising}). In particular, we obtain a positive answer to the Fu--Namikawa conjecture (\cite[Conjecture 2]{Fu-Namikawa}) 
in the following context: two crepant resolutions of a given symplectic space lead to deformation equivalent irreducible symplectic orbifolds (see Corollary \ref{crepant}).

On the other hand, the global Torelli theorem (Theorem \ref{main}) and our results on wall divisors allow us to generalize the notion of lattice mirror symmetry to the orbifold setting in Section \ref{mir}. 
In \cite[Section 6]{Huybrechts2}, Huybrechts proposes a definition of mirror symmetry at the level of the following period domain: $$\dG_\Lambda:=\left\{\left.(\alpha,\beta+ix)\in
(\Lambda\otimes\C)^2\ \right|\ \alpha^2=0,\ \alpha\cdot\overline{\alpha}>0,\ \alpha\cdot x=0,\  x ^2>0
\right\}.$$ In Theorem \ref{TorelliMirror}, using our results on the Kähler cone and the global Torelli theorem, we are able to provide an isomorphism between a sub-period domain of $\dG_\Lambda$ and the moduli space $\widetilde{\mathcal{M}_{\Lambda}}:=\left\{(X,\varphi,\omega_X,\sigma_X,\beta)\right\}$, 
which parametrizes 
deformation equivalent marked irreducible symplectic orbifolds of Beauville--Bogomolov lattice $\Lambda$ endowed with a Kähler class $\omega_X$, a holomorphic 2-form $\sigma_X$ and a class $\beta\in H^2(X,\R)$. Theorem \ref{TorelliMirror} makes it possible to define a mirror symmetry as an involution on a sub-moduli space of $\widetilde{\mathcal{M}_{\Lambda}}$. Therefore, this allows us to define the symmetric mirror of a precise irreducible symplectic orbifold. This is a new approach compared to what was done in \cite{Dolgachev} and \cite{Camere-2016}, where the mirror symmetry acts on a set of moduli spaces.
This is an improvement of results in \cite[Section 4]{Franco-Jardim-Menet} and is also new in the setting of irreducible symplectic manifolds.

We recall below the construction of an irreducible symplectic orbifold which occupies an important role in the literature (see \cite[Section 13, table 1, line I.2]{Fujiki-1983}, \cite[Corollary 5.7]{Marku-Tikho},
\cite{Menet-2015}, \cite[Section 3.2]{Menet-2020}). We will denote it by $M'$. Let $X$ be a $K3^{[2]}$-type hyperkähler manifold, carrying a symplectic involution $\iota$. We define $M'$ to be a crepant resolution in codimension 2 of the quotient $X/\iota$ (see Example \ref{exem}). In particular, we prove that this orbifold is deformation equivalent to a Fujiki orbifold (see Proposition \ref{FujikiM'}) and to the Markushevich--Tikhomirov orbifold (see Proposition \ref{rappel}).

In a sequel of this paper, we will provide a full description of the wall divisors on $M'$ \cite{Ulrike2}. This is an
application of the results from Section \ref{Kähler}, and will make it possible to determine the Kähler cone of
irreducible symplectic orbifolds of this deformation type explicitly.

The paper is organized as follows. In Section \ref{reminders}, we provide some reminders for important results
on irreducible symplectic orbifolds. Section \ref{proof} is devoted to the proof of Theorem \ref{main} with
applications to the singularities in Section \ref{applising}. In Section \ref{nonseparated}, we provide an
example of non-separated points in a moduli space of marked irreducible symplectic orbifolds. 
In Section 4, we generalize Boucksom's theorem (\cite[Théorème 1.2]{Boucksom-2001})
to describe the Kähler cone of an irreducible symplectic orbifold. Moreover, we
extend existing results on wall divisors from the smooth setting to the orbifold
one.
 Finally, in Section \ref{mir}, we apply our results to define a mirror symmetry.
~\\

\textbf{Acknowledgements}: We would like to thank 
Lie Fu and Giovanni Mongardi for helpful discussions. We are very grateful to the Second Japanese-European Symposium on Symplectic Varieties and Moduli Spaces where our collaboration was initiated. The first author has been financed 
by the ERC-ALKAGE grant No. 670846 and by the PRCI SMAGP (ANR-20-CE40-0026-01).
The second author was a member of the Institute for Theoretical Studies at ETH Zürich
(supported by Dr.~Max R\"ossler, the Walter Haefner Foundation and the ETH Z\"urich
Foundation).
\section{Reminder on holomorphically symplectic orbifolds}\label{reminders}
\subsection{Definition}\label{basicdef}
First, we briefly recall the definition of a symplectic orbifold used in this paper. 
\begin{defi}\label{VV}
A \emph{$n$-dimensional orbifold} is a connected paracompact complex space $X$ such that for every point $x\in X$, there exists an open neighborhood $U$ and a triple $(V,G,\pi)$ such that $V$ is an open subset of $\C^n$, $G$ is a finite automorphisms group of $V$ and $\pi:V\rightarrow U$ is the composition of the quotient map $V\rightarrow V/G$ and an isomorphism $V/G\simeq U$. 
The quadruple $(U,V,G,\pi)$ is called a \emph{local uniformizing system} of $x$.
\end{defi}
\begin{rmk}\label{normal}
Note that an orbifold is normal as a complex space (see for instance \cite[Th\'eor\`eme 4]{Cartan:Quotient}).
\end{rmk}

Let $X$ be an orbifold. 
A \emph{smooth differential $k$-form} $\omega$ on $X$ is a $C^{\infty}$ differential $k$-form on $X_{reg}:=X\smallsetminus \Sing X$ such that for all local uniformizing systems $(U,V,G,\pi)$, $\pi^*(\omega_{|U_{reg}})$ extends to a $C^{\infty}$-differential $k$-form on $V$ (see \cite[Section 2]{Menet-2020} and \cite[Section 1 and 2]{Blache} for more details). 

A smooth differential form $\omega$ on $X$ is called \emph{Kähler} if it is a Kähler form on $X_{reg}$ and such
that for all local uniformizing systems $(U,V,G,\pi)$, the pullback $\pi^*(\omega_{|U_{reg}})$ extends to a Kähler form on $V$. A \emph{Kähler orbifold} is an orbifold which admits a Kähler form.
\begin{rmk}
An orbifold is Kähler if and only if it is Kähler as a complex space (cf. \cite[Definition 1.2, 4.1 and Remark 4.2]{Fujiki-1978} for the definition of Kähler complex space and \cite[pages 793-795]{CF-Zaffran} for the proof of the equivalence).
\end{rmk}

We denote by $\mathcal{K}_X$ the \emph{Kähler cone} of $X$, which is the set of all De Rham classes of Kähler forms on $X$.

\begin{defi}[\cite{Menet-2020}, Definition 3.1]\label{def}
A compact K\"ahler orbifold $X$ is called \emph{primitively symplectic} if:
\begin{itemize}
\item[(1)]
$X_{reg}:=X\smallsetminus \Sing X$ is endowed with a non-degenerate holomorphic 2-form, which is unique up to scalar.
\item[(2)]
$\codim \Sing X\geq 4$.
\end{itemize}
If moreover $X_{reg}$ is simply connected, $X$ is said an \emph{irreducible symplectic orbifold}.
\end{defi}
We refer to \cite[Section 6]{Campana-2004}, \cite[Section 3.1]{Menet-2020} and \cite[Section 3.1]{Fu-Menet} for discussions about this definition. 
Condition (2) is equivalent to having terminal singularities (\cite{Namikawa:note}). According to \cite{BCHM}, we can always partially solve the algebraic symplectic singularities in order to obtain terminal singularities; however it is not known if quotient singularities are stable under this process.
\begin{ex}[{\cite[Section 3.2]{Menet-2020}}]\label{exem}
Let $X$ be a hyperkähler manifold deformation equivalent to a Hilbert scheme of 2 points on a K3 surface and $\iota$ a symplectic involution on $X$.
By \cite[Theorem 4.1]{Mongardi-2012}, $\iota$ has 28 fixed points and a fixed K3 surface $\Sigma$. 
We denote by $M'$ the blow-up of $X/\iota$ at the image of $\Sigma$. The orbifold $M'$ is irreducible
symplectic (see \cite[Proposition 3.8]{Menet-2020}).
\end{ex}
\begin{defi}\label{Nikulin}
An orbifold $M'$ obtained from the previous construction is called a \emph{Nikulin} orbifold. An orbifold deformation equivalent to a Nikulin orbifold is called of \emph{Nikulin type}.
\end{defi}
\begin{ex}\label{exem2}
 A similar construction of an irreducible symplectic orbifold (starting from a generalized Kummer
  variety) can also be found in \cite[Proposition
    3.8]{Menet-2020}; we denote this orbifold by $K'$. See \cite[Section 13]{Fujiki-1983} and \cite[Section 5]{Fu-Menet} for further examples of primitively
  symplectic orbifolds.
\end{ex}
\subsection{Beauville--Bogomolov form}
Let $X$ be an irreducible symplectic orbifold.
As explained in \cite[Section 3.4]{Menet-2020}, $H^2(X,\Z)$ can be endowed with an indivisible and integral bilinear
form of signature $(3,b_2(X)-3)$, called the \emph{Beauville--Bogomolov form} and denoted by $(\,,\,)_q$. 
Let $\sigma_X$ be a holomorphic 2-form of $X$ with $\sigma_X^{n}\overline{\sigma}_X^{n}=1$; the quadratic form $q$ associated to $(\,,\,)_q$ is obtained, for all $\alpha\in H^2(X,\Z)$, by:
$$q(\alpha):=\lambda\left[\frac{n}{2}(\sigma_X\overline{\sigma}_X)^{n-1}\alpha^2+(1-n)\left(\sigma_X^{n-1}\overline{\sigma}_X^{n}\alpha\right)\cdot\left(\sigma_X^{n}\overline{\sigma}_X^{n-1}\alpha\right)\right],$$
with $\lambda\in \R$ such that $(\,,\,)_q$ is integral and primitive.
\begin{ex}[\cite{Menet-2015}, Theorem 2.5; \cite{Kapfer-Menet}, Theorem 1.1]\label{exempleBB}
The irreducible symplectic orbifolds $M'$ and $K'$ of Example \ref{exem} and \ref{exem2} have the following Beauville--Bogomolov lattice respectively:
$$H^{2}(M',\Z)\simeq U(2)^3\oplus E_8(-1)\oplus (-2)^2\ \text{and}\ H^2(K',\Z)\simeq U(3)^3\oplus \begin{pmatrix} -5 & -4\\ -4 & -5\end{pmatrix}.$$
\end{ex} 
\begin{rmk}\label{dualclass}
Let $\beta\in H^{2n-1,2n-1}(X,\Q)$.
We can associate to $\beta$ its \emph{dual class} $\beta^{\vee}\in H^{1,1}(X,\Q)$, defined as follows. By \cite[Corollary 2.7]{Menet-2020} and since the Beauville--Bogomolov form is integral and non-degenerate (see \cite[Theorem 3.17]{Menet-2020}), we can find $\beta^{\vee}\in H^{2}(X,\Q)$ such that for all $\alpha\in H^2(X,\C)$:
 $$(\alpha,\beta^{\vee})_q=\alpha\cdot \beta,$$ 
where the dot on the right hand side is the cup product. Since $(\beta^{\vee},\sigma_X)_q=\beta\cdot \sigma_X=0$, we have $\beta^{\vee}\in H^{1,1}(X,\Q)$.
\end{rmk}
\subsection{Period map}\label{per}
Let $\Lambda$ be an abstract lattice of signature $(3,\rk \Lambda-3)$. 
A \emph{marking} of a primitively symplectic orbifold $X$ is an isometry $\varphi: H^{2}(X,\Z)\rightarrow \Lambda$.
We denote by $\mathcal{M}_{\Lambda}$ the set of isomorphism classes of marked irreducible symplectic orbifolds $(X,\varphi)$ with $\varphi:H^2(X,\Z)\rightarrow\Lambda$. As explained in \cite[Section 3.5]{Menet-2020}, this set can be endowed with a complex structure, such that the \emph{period map}:
$$\xymatrix@R0cm@C0.5cm{\ \ \ \ \ \ \ \ \mathscr{P}:& \mathcal{M}_{\Lambda}\ar[r]& \mathcal{D}_{\Lambda}\\
&(X,\varphi)\ar[r]&\varphi(\sigma_X)}$$
is a local isomorphism with $\mathcal{D}_\Lambda:=\left\{\left.\alpha\in
\mathbb{P}(\Lambda_{\C})\ \right|\ \alpha^2=0,\ \alpha\cdot\overline{\alpha}>0\right\}$.
The complex manifold $\mathcal{M}_{\Lambda}$ is called \emph{the moduli space of marked primitively symplectic
  orbifolds of Beauville--Bogomolov lattice $\Lambda$}.
When there is no ambiguity on $\Lambda$, we will write $\dD$ instead of $\dD_\Lambda$.

Let $\mathscr{X}\overset{f}{\rightarrow} S$ be a deformation of $X$ by primitively symplectic orbifolds, where
the base $S$ is simply connected. We denote $o:=f(X)$. By \cite[Theorem 3.17]{Menet-2020}, we can see $S$ as an analytic subset of $\mathcal{M}_{\Lambda}$.
Moreover, by \cite[Corollary 3.12]{Menet-2020}, there exists a natural isomorphism $u_s:
H^*(\mathscr{X}_s,\C)\rightarrow H^*(X,\C)$ which respects the cup product. Let $\varphi$ be a marking for $X$. The period map restricted to $S$ has the following expression: 
$$
\xymatrix@R0pt{
\mathscr{P}:S\ar[r]& \mathcal{D} \\
\ \ \ \ s\ar[r] &\varphi\circ u_{s}(\sigma_{\mathscr{X}_s}),
}$$
For instance, let $f:\mathscr{X}\rightarrow \Def(X)$ be the Kuranishi deformation of $X$ (see for instance \cite[Remark 3.4]{Fujiki-1983}). By \cite[Proposition 3.16]{Menet-2020}, there exists an open neighborhood $U$ of $o\in \Def(X)$, such that, for all $t\in U$, the fiber $\mathscr{X}_t:=f^{-1}(t)$ is a primitively symplectic manifold. 
Moreover, we can shrink $U$ to a smaller open neighborhood of $o$ such that $\mathscr{P}: U\rightarrow \mathcal{P}(U)$ is an isomorphism.
\subsection{Global Torelli theorem}\label{GTT}
The moduli space $\mathcal{M}_{\Lambda}$ introduced in the previous section is a non-separated manifold, however by \cite[Corollary 3.25]{Menet-2020}, there exists a \emph{Hausdorff reduction} $\overline{\mathcal{M}_{\Lambda}}$ of $\mathcal{M}_{\Lambda}$, such that the period map $\mathscr{P}$ factors through $\overline{\mathcal{M}_{\Lambda}}$:
$$\xymatrix{\mathcal{M}_{\Lambda}\ar@/^1pc/[rr]^{\mathscr{P}}\ar@{->>}[r]& \overline{\mathcal{M}_{\Lambda}}\ar[r]& \dD.}$$
Moreover, two points in $\mathcal{M}_{\Lambda}$ map to the same point in $\overline{\mathcal{M}_{\Lambda}}$ if and only if they are non-separated in $\mathcal{M}_{\Lambda}$. Finally, we can recall the \emph{global Torelli Theorem}.

\begin{thm}[\cite{Menet-2020}, Theorem 1.1]\label{mainGTTO}
Let $\Lambda$ be a lattice of signature $(3,b-3)$, with $b\geq3$. Assume that $\mathcal{M}_{\Lambda}\neq\emptyset$ and let $\mathcal{M}_{\Lambda}^{\circ}$ be a connected component of $\mathcal{M}_{\Lambda}$. Then the period map:
$$\mathscr{P}: \overline{\mathcal{M}_{\Lambda}}^{\circ}\rightarrow \mathcal{D}$$
is an isomorphism. 
\end{thm}

\subsection{Positive cone}\label{posi}
The \emph{positive cone} of $X$, which we denote by $\mathcal{C}_X$, is the connected component of $$\left\{\left.\alpha\in H^{1,1}(X,\R)\right|\ q(\alpha)>0\right\}$$ containing the K\"ahler cone (see Section \ref{basicdef}). The following result about the elements of the positive cone will be used several times in this paper.

\begin{prop}[\cite{Menet-2020}, Proposition 5.5]\label{oldprop}
Let $X$ be an irreducible symplectic orbifold, and $\alpha\in \mathcal{C}_X$ be a very general element of the positive cone, i.e. $\alpha$ is contained in the complement of countably many nowhere dense closed subsets. Then there exist two flat proper families $\mathcal{X}\rightarrow S$ and $\mathcal{X}'\rightarrow S$ of irreducible symplectic orbifolds over a one-dimensional disk $S$ and a bimeromorphism $F:\mathcal{X}'\dashrightarrow\mathcal{X}$ compatible with the projection to $S$, such that $F$ induces an isomorphism $\mathcal{X}'_{|S\smallsetminus\left\{0\right\}}\simeq \mathcal{X}_{|S\smallsetminus\left\{0\right\}}$, the special fiber $\mathcal{X}_0$ is isomorphic to $X$, and $F^*\alpha$ is a K\"ahler class on $\mathcal{X}'_0$.
\end{prop}
In this proposition by $F^*(\alpha)$ we refer to the following.
If $\mathcal{X}'\leftarrow \mathcal{Z} \rightarrow \mathcal{X}$ resolves the bimeromorphic map $F$, we can consider $\Gamma=\Ima (\mathcal{Z}_0\rightarrow X\times \mathcal{X}'_0)$, with $\mathcal{Z}_0$ the fiber over $0$ of $\mathcal{Z}\rightarrow S$. We consider $p: \Gamma\rightarrow X$ and $p': \Gamma\rightarrow \mathcal{X}'_0$ the projections.
Then, we set $F^*(\alpha):=p'_*(\Gamma\cdot p^*(\alpha))$.

\begin{remark}
  In Section \ref{sec:wall-divs} we will introduce the concept of wall divisor. This will make it possible to replace the
  condition that $\alpha$ is very general by a much more explicit condition, which can be verified for
  individual elements $\alpha \in \dC_X$ (compare Corollary 
  \ref{cor:alphamovedtoK}).
\end{remark}

\subsection{Twistor space}\label{Twist}
Let $\Lambda$ be a lattice of signature $(3,\rk\Lambda-3)$. We denote by "$\cdot$" its bilinear form.
A \emph{positive three-space} is a subspace $W\subset \Lambda\otimes\R$ such that $\cdot_{|W}$ is positive
definite.
\begin{defi}\label{def:twistorline}
  For any positive three-space $W$, we define the associated \emph{twistor line} $T_W$ as
$$T_W:=\mathcal{D}\cap \mathbb{P}(W\otimes\C).$$
A twistor line is called \emph{generic} if $W^{\bot}\cap \Lambda=0$. A point of $\alpha\in \mathcal{D}$ is called
\emph{very general} if $\alpha^{\bot}\cap \Lambda=0$.
\end{defi}
\begin{defi}
Let $W$ be a positive three-space endowed with an orientation. Let $p:\Lambda\rightarrow W$ be the orthogonal projection. 
We set:
$$O^+(\Lambda):=\left\{\left.g\in O(\Lambda)\right|\ p\circ g_{|W}\ \text{respects the orientation of}\ W\right\}.$$
\end{defi}
\begin{thm}[\cite{Menet-2020}, Theorem 5.4]\label{Twistor}
  Let $(X,\varphi)$ be a marked irreducible symplectic orbifold with $\varphi:H^2(X,\Z)\rightarrow \Lambda$. Let $\alpha$ be a Kähler class on $X$, and $W_\alpha\coloneqq\Vect_{\R}(\varphi(\alpha),$ $\varphi(\Rea \sigma_X),\varphi(\Ima \sigma_X))$. 
Then:
\begin{itemize}
\item[(i)]
There exists a metric $g$ and three complex structures (see \cite[Section 5.1]{Menet-2020} for the definition) $I$, $J$ and $K$ in quaternionic relation on $X$ such that:
$$\alpha= \left[g(\cdot,I\cdot)\right]\ \text{and}\ g(\cdot,J\cdot)+ig(\cdot,K\cdot)\in H^{0,2}(X).$$
\item[(ii)]
There exists a deformation of $X$: 
$$\mathscr{X}\rightarrow T(\alpha)\simeq\mathbb{P}^1,$$ such that the period map
$\mathscr{P}:T(\alpha)\rightarrow T_{W_\alpha}$ provides an isomorphism. Moreover, for each $s=(a,b,c)\in
\mathbb{P}^1$, the associated fiber $\mathscr{X}_s$ is an orbifold diffeomorphic to $X$ endowed with the
complex structure $aI+bJ+cK$ (identifying $\bP^1$ with $S^2=\{(a,b,c)\in \bR^3 \,|\, a^2+b^2+c^2 = 1\}$).
\end{itemize}
\end{thm}
\begin{rmk}
Note that if the irreducible symplectic orbifold $X$ of the previous theorem is endowed with a marking then all the fibers of $\mathscr{X}\rightarrow T(\alpha)$ are naturally endowed with a marking.
%
\end{rmk}

For the statement of the next lemma, we recall the definition of a parallel transport operator.
\begin{defi}\label{transp}
Let $X_1$ and $X_2$ be two irreducible symplectic orbifolds. An isometry $f:H^{2}(X_{1},\Z)\rightarrow H^{2}(X_{2},\Z)$ is called a \emph{parallel transport operator} if there exists a deformation $s:\mathcal{X}\rightarrow B$, two points $b_{i}\in B$, two isomorphisms $\psi_{i}:X_{i}\rightarrow \mathcal{X}_{b_{i}}$, $i=1,2$ and a continuous path $\gamma:\left[0,1\right]\rightarrow B$ with $\gamma(0)=b_{1}$, $\gamma(1)=b_{2}$ and such that the parallel transport in the local system $R^2s_{*}\Z$ along $\gamma$ induces the morphism $\psi_{2*}\circ f\circ\psi_{1}^{*}: H^{2}(\mathcal{X}_{b_{1}},\Z)\rightarrow H^{2}(\mathcal{X}_{b_{2}},\Z)$.
\end{defi}
\begin{rmk}
Note that a deformation of irreducible symplectic orbifolds is always locally trivial (see \cite[Proposition 3.10]{Menet-2020}).
\end{rmk}
\begin{defi}\label{Mon}
Let $X$ be an irreducible symplectic orbifold. The group of parallel transport operators from $X$ to itself is a subgroup of $O(H^2(X,\Z))$, called the \emph{monodromy group} of $X$, and denoted by $\Mon^2(X)$. 
\end{defi}
\begin{lemme}\label{connected}
Let $\mathscr{M}_{\Lambda}^{\circ}$ be a connected component of the moduli space $\mathscr{M}_{\Lambda}$.
Two marked irreducible symplectic orbifolds $(X,\varphi)$, $(Y,\psi)$ in the  $\mathscr{M}_{\Lambda}^{\circ}$ are connected by twistor spaces. 
That is: There exists a sequence of generic twistor spaces $f_i:\mathscr{X}_i\rightarrow \mathbb{P}_i^1$ with
$(x_i,x_i')\in\mathbb{P}_i^1\times\mathbb{P}_i^1$, $i\in \left\{0,...,k\right\}$, $k\in \mathbb{N}$ such that 
$$f^{-1}_0(x_0)\simeq (X,\varphi),\ f^{-1}_i(x_{i}')\simeq f^{-1}_{i+1}(x_{i+1})\ \text{and}\ f^{-1}_{k}(x_{k}')\simeq (Y,\psi),$$
for all $0\leq i\leq k-1$, which induces $\psi^{-1}\circ \phi$ as parallel transport operator.
\end{lemme}

\begin{proof}We prove the Lemma by distinguishing two cases.
  
\textbf{First case}: We assume that $(X,\varphi)$ and $(Y,\psi)$ are very general (that is $\Pic X=0$ and $\Pic Y=0$). 
By \cite[Proposition 3.7]{Huybrechts12}  the period domain $\dD_{\Lambda}$
is connected by generic twistor lines. Note that the proof of \cite[Proposition 3.7]{Huybrechts12} in fact shows
that the twistor lines can be chosen in such a way that they intersect in very general points of $\dD_{\Lambda}$.
In particular, we can connect  $\mathscr{P}(Y,\psi )$ and $\mathscr{P}(X,\phi)$ by
such generic twistor lines in $\dD_{\Lambda}$.
  Since for a very general element $(\widetilde{X},\widetilde{\phi})$ of $\dM_\Lambda$ we have
  $\dK_{\widetilde{X}}=\dC_{\widetilde{X}}$, by \cite[Corollary 5.6]{Menet-2020}, Theorem \ref{Twistor} shows that all these twistor lines can be lifted to
twistor spaces. By \cite[Proposition 3.22]{Menet-2020} the period map $\cP$ is injective on the set of very
general points in $\dD_\Lambda$. Therefore, all these twistor spaces intersect and connect $(X,\varphi)$ to $(Y,\psi)$.

\textbf{Second case}: If $(X,\varphi)$ is not very general, we consider 
a very general Kähler class $\alpha$. Then the associated twistor space $\mathscr{X}\rightarrow T(\alpha)$
has a fiber which is a very general marked irreducible symplectic orbifold. Hence we are back to the first case.
\end{proof}

\section{Hodge version of the global Torelli theorem}
\subsection{Proof of Theorem \ref{main}}\label{proof}
\begin{lemme}\label{orbiornotorbi}
Let $X$ be an orbifold and $f:\widetilde{X}\rightarrow X$ the blow-up of $X$ in a smooth subvariety $Y$ with $\codim_X(Y)\geq2$.
Let $E$ be the exceptional divisor of $f$.  Then:
$$H^2(\widetilde{X},\R)\simeq H^2(X,\R)\oplus\R\left[E\right].$$
\end{lemme}
\begin{proof}
The map $f$ is the blow-up of $X$ in an analytic subset $Y$.
Let $U:=X\smallsetminus Y$. Since $\codim Y\geq2$, we have
$H^2(X,\R)\simeq H^2(U,\R)$ (see for instance \cite[Lemma 1.6]{Fujiki-1983}).
Then the commutative diagram
$$\xymatrix@C15pt{
 H^2(X,\R)\eq[d]\ar[r]^{f^*}&H^2(\widetilde{X},\R)\ar[dl]\\
 H^2(U,\R)&
    }$$
shows that $f^*:H^2(X,\R)\rightarrow H^2(\widetilde{X},\R)$ is an injection.
We have $f^*(H^2(X,\R))\oplus\R\left[E\right]\subset H^2(\widetilde{X},\R)$.
Let $V:=\widetilde{X}\smallsetminus \Sing \widetilde{X}$.
Still because of \cite[Lemma 1.6]{Fujiki-1983}, $H^2(\widetilde{X},\R)\simeq H^2(V,\R)$.
Then we conclude with the following exact sequence:
$$\xymatrix@C15pt{
&&H^2(X,\R)\eq[d]\\
 H^2(V,V\smallsetminus E,\R)\eq[d]\ar[r]&H^2(V,\R)\ar[r]&H^2(V\smallsetminus E,\R)\\
 H^0(E\cap V,\R), & &
    }$$
where the isomorphisms are given respectively by Thom's isomorphism and \cite[Lemma 1.6]{Fujiki-1983}. 
\end{proof}
\begin{lemme}\label{biracodim}
Let $f:X\dashrightarrow X'$ be a bimeromorphic map between two primitively symplectic orbifolds. Then $f$ is an isomorphism in codimension 1.
\end{lemme}
\begin{proof}
Since $X$ is normal (Remark \ref{normal}), $f$ is well defined on $U_0\subset X$, with $\codim X\smallsetminus
U_0\geq2$. Moreover, we can look at $U:=U_0\smallsetminus \Sing U_0$. We also have $\codim X\smallsetminus
U\geq2$. Let $U'$ be the smooth locus of $X'$.
First, observe that $f(U)$ is contained in $U'$. Indeed, if this was not the case, $f|_U$ would be a crepant resolution of some singularities of $X'$ which is impossible since $X'$ has only terminal singularities (see \cite[Corollary 1]{Namikawa:note}).

Consider the subset $Y\subset U$ consisting of all points $x$ such that
$f^{-1}(f(x))\neq\left\{x\right\}$. Let $\sigma'$ be the unique non-degenerate holomorphic 2-form on $U'$,
then $f^*(\sigma')$ is also the unique non-degenerate holomorphic 2-form on $U$ (up to a scalar) because
$f^*(\sigma')$ extends to all $X$ (see for instance \cite[Lemma 2.1]{Fujiki-1983}). Since $f^*(\sigma')$ is
non-degenerate, $Y$ cannot be a divisor, otherwise there would be positive dimensional fibers,
and they would lie in the set of degeneration of $f^*(\sigma')$ which is impossible. Finally, we set $V:= U\cap f^{-1}(U')\smallsetminus Y$ which verifies $\codim X\smallsetminus V\geq 2$ and $f:V\rightarrow X'$ is an injective bimeromorphic morphism. With the same argument, there exists $V'\subset X'$ such that 
$\codim X'\smallsetminus V'\geq 2$ and $f^{-1}:V'\rightarrow X$ is an injective bimeromorphic morphism. Then
$f:V\cap f^{-1}(V')\rightarrow V'\cap f(V)$ is an isomorphism, and, using the same arguments as above, we have $\codim
X\smallsetminus V\cap f^{-1}(V')\geq2$ and $\codim X\smallsetminus V'\cap f(V)\geq2$.
\end{proof}
\begin{prop}\label{mainprop}
Let $f:X'\dashrightarrow X$ be a bimeromorphic map between two primitively symplectic orbifolds. If $\alpha\in H^2(X,\R)$ is a class such that $\alpha\cdot \left[C\right]>0$ and $f^*(\alpha)\cdot \left[C'\right]>0$ for all rational curves $C\in X$ and $C'\in X'$, then $f$ extends to an isomorphism.
\end{prop}
\begin{proof}
The proof of this proposition is an adaptation of the proof of \cite[Proposition 2.1]{Huybrechts2003}.
By Hironaka's theorem (see for instance \cite{Bierstone-Milman}) there exists a sequence of blow-ups along smooth centers $\pi:Z\rightarrow X$ resolving $f$. In particular, we can find a non-zero effective divisor $\sum n_iE_i$ such that $-\sum n_iE_i$ is $\pi$-ample. We have the following commutative diagram:
$$\begin{array}{cccccc}
&&Z&&\\
&\pi'\swarrow&&\searrow\pi&\\
&X'~-&\stackrel{f}{-}&\to~X.&
\end{array}$$
By Lemma \ref{biracodim}, any exceptional divisors $E_i$ for $\pi$ is also an exceptional divisor for $\pi'$.

Let $\beta\in H^2(X,\R)$. Using Lemma \ref{orbiornotorbi} there exists $\gamma\in H^2(X',\R)$ such that:
$$\pi^*\beta=\pi'^*\gamma+\sum_ia_i\left[E_i\right].$$
Let $W\subset X$ and $W'\subset X'$ such that $\codim X\smallsetminus W\geq 2$, $\codim X'\smallsetminus W'\geq 2$ and the restriction $f:W\rightarrow W'$ is an isomorphism. We have a commutative diagram:
$$\xymatrix@C15pt{
 &\ar[dl]_{\pi}\pi^{-1}(W)\ar[dr]^{\pi'}&\\
 W\ar[rr]^{f}& & W'.
    }$$
 The commutativity of this diagram and \cite[Lemma 1.6]{Fujiki-1983} imply that $\gamma=f^*\beta$. That is:
 \begin{equation}
\pi^*\beta=\pi'^*(f^*(\beta))+\sum_ia_i\left[E_i\right].
 \label{HuyImportant}
 \end{equation}
 Taking $\alpha\in H^2(X,\R)$ as in the statement of the proposition,
exactly as Huybrechts in \cite[Proof of Proposition 2.1]{Huybrechts2003}, we can see that the $a_i$ vanish in this case: 
$$\pi^*\alpha=\pi'^*(f^*(\alpha)).$$
We conclude as Huybrechts proving that a rational curve $C$ is contracted by $\pi$ if and only if it is
contracted by $\pi'$  regarding the intersections $\pi^*\alpha\cdot C$ and $\pi'^*(f^*(\alpha))\cdot C$. 
Since all the exceptional divisors are covered by rational curves (our resolution is a sequence of blow-ups), all contractions $\pi_{|E_i}$ and $\pi'_{|E_i}$ coincide. So $f$ extends to an isomorphism.
\end{proof}
\begin{lemme}\label{isometry}
Let $f:X'\dashrightarrow X$ be a bimeromorphic map between two primitively symplectic orbifolds. By Lemma \ref{biracodim}, there exist two open sets $j:W\hookrightarrow X$ and $j':W'\hookrightarrow X'$ such that $\codim X\smallsetminus W\geq2$, $\codim X'\smallsetminus W'\geq2$ and the restriction $f:W'\rightarrow W$ is an isomorphism. Then the map induced by \cite[Lemma 1.6]{Fujiki-1983}:
$$f^*:H^2(X,\R)\stackrel{j^*}{\widetilde{\rightarrow}}H^2(W,\R)\simeq H^2(W',\R)\stackrel{\left(j'^*\right)^{-1}}{\widetilde{\rightarrow}}H^2(X',\R),$$
is an isometry with respect to the Beauville--Bogomolov form.
\end{lemme}
\begin{proof}
The same arguments as in \cite[(I.6.2) Proposition]{OG-1997} yield the proof.
\end{proof}
\begin{prop}\label{mainprop2}
Let $X$ and $X'$ be irreducible symplectic orbifolds and $f:X'\dashrightarrow X$ a bimeromorphism. Then there exist two flat proper families $\mathcal{X}\rightarrow S$ and $\mathcal{X}'\rightarrow S$ over a one-dimensional disk $S$ with the following  properties:
\begin{itemize}
\item[(i)]
The special fibers are $\mathcal{X}_0\simeq X$ and $\mathcal{X}'_0\simeq X'$.
\item[(ii)]
There exists a bimeromorphism $F:\mathcal{X}'\dashrightarrow\mathcal{X}$ which is an isomorphism over $S\smallsetminus \left\{0\right\}$, i.e. $F:\mathcal{X}'_{|S\smallsetminus \left\{0\right\}}\iso\mathcal{X}_{|S\smallsetminus \left\{0\right\}}$, and which coincides with $f$ on the special fiber, i.e. $F_0=f$.
\end{itemize}
\end{prop}
\begin{proof}
Using Lemmas \ref{orbiornotorbi} and \ref{biracodim} and Propositions \ref{mainprop} and \ref{oldprop}, this result can be proved exactly as Huybrechts did in \cite[proof of Theorem 2.5]{Huybrechts2003}.
\end{proof}
\begin{proof}[Proof of Theorem \ref{main}]
Statement (i) is given by Proposition \ref{mainprop2} and \cite[Corollary 5.11]{Menet-2020}. 
In the following, we show statement (ii).  Clearly, if there exists and isomorphism $\tilde f\colon X \to X'$, then
$\tilde f_*$
maps some Kähler class to a Kähler class. Hence, we only need to prove the reverse implication.

Let $f:H^2(X,\Z)\rightarrow H^2(X',\Z)$ be a parallel transport
which is a Hodge isometry. 
By Theorem \ref{mainGTTO} and Proposition \ref{mainprop2}, there exist two flat proper families $\mathcal{X}\rightarrow S$ and $\mathcal{X}'\rightarrow S$ over a one-dimensional disk $S$ such that $\mathcal{X}_0\simeq X$ and $\mathcal{X}'_0\simeq X'$ and $F:\mathcal{X}'\dashrightarrow\mathcal{X}$ which is an isomorphism over $S\smallsetminus \left\{0\right\}$.
Let $\mathcal{X}\leftarrow\mathcal{Z}\rightarrow\mathcal{X}'$ be a resolution of $F$ obtained by a sequence of
blow-ups in smooth loci (Hironaka's theorem). Let $\mathcal{Z}_0$ be the fiber of $\mathcal{Z}\rightarrow S$
over $0$, we consider the cycle $\Gamma:=\Image (\mathcal{Z}_0\rightarrow X\times X')$ which decomposes as
$\Gamma=Z+\sum Y_i$. As explained in \cite[Remark 3.24]{Menet-2020}, we know that $Z$ is the graph of a
bimeromorphism between $X$ and $X'$. Moreover $\left[\Gamma\right]_*=f$, where
$\left[\Gamma\right]_*(x):=\pi'_*([\Gamma]\cdot \pi^*(x))$ for all $x\in H^2(X,R)$. Assume that there exists $\alpha\in H^{1,1}(X,\bR)$ a K\"ahler class such that $\beta=f(\alpha)=\left[\Gamma\right]_*(\alpha)$ is a K\"ahler class too. We are now in the same situation as the one in the proof of \cite[Theorem 2.5]{Huybrechts2003} and with the same argument, we prove that $\pi'(Y_i)$ has codimension at least 2 for all $i$. Hence $f=\left[\Gamma\right]_*=\left[Z\right]_*$. Finally, we conclude the proof using Proposition \ref{mainprop}.
\end{proof}
\subsection{Applications to the singularities of an irreducible symplectic orbifold}\label{applising}
\subsubsection*{Two bimeromorphic irreducible symplectic orbifolds have the same singularities}
We note the following corollary of Proposition \ref{mainprop2}. 
\begin{cor}
Let $X$ and $X'$ be two bimeromorphic irreducible symplectic orbifolds. Then, $X$ and $X'$ have the same singularities.
\end{cor}
\begin{proof}
We have seen with Proposition \ref{mainprop2} that $X$ and $X'$ are deformation equivalent. However,  quotient singularities of codimension $\geq3$ are rigid under deformation (see \cite[Lemma 3.3]{Fujiki-1983}).
\end{proof}
\subsubsection*{Uniqueness of the crepant resolution in codimension 2}
The following corollary of Proposition \ref{mainprop2} can be seen as a partial answer to \cite[Conjecture 2]{Fu-Namikawa}. 
By \cite{BCHM}, we know that symplectic singularities in codimension 2 have a partial crepant resolution. Is this partial crepant resolution unique? 
\begin{defi}
Let $X$ be an orbifold. Let $\Sing_2 X$ be the union of the irreducible components of $\Sing X$ of codimension 2. A partial resolution in codimension 2, $r:\widetilde{X}\rightarrow X$ of $X$, is a proper bimeromorphic map such that:
\begin{itemize}
\item[(i)]
the restriction $r:r^{-1}\left(X\smallsetminus \Sing_2 X\right)\rightarrow X\smallsetminus \Sing_2 X$ is an isomorphism;
\item[(ii)]
$\codim \Sing \widetilde{X} \geq 3$.
\end{itemize}
\end{defi}
\begin{cor}\label{crepant}
Let $X$ and $X'$ be two irreducible symplectic orbifolds. We assume that $X$ and $X'$ are two partial
resolutions in codimension 2 of the same orbifold $Y$. Then $X$ and $X'$ are deformation equivalent.
\end{cor}
\begin{proof}
Indeed, $X$ and $X'$ are bimeromorphic. Then Proposition \ref{mainprop2} concludes the proof.
\end{proof}
The corollary above may help to classify irreducible symplectic orbifolds in dimension 4. We provide an example below.
\begin{ex}
In \cite[Section 13]{Fujiki-1983}, Fujiki constructs the following symplectic orbifold. He considers $S$ a K3 surface and $i$ a symplectic involution on $S$.
Let $s_2$ be the reflection on $S\times S$ acting by $s_2(x,y)=(y,x)$.
Let $G=\left\langle s_2,i\right\rangle$ be the automorphism group on $S\times S$ with $i(x,y)=(i(x),i(y))$. Fujiki defined $S(C_2)^{[2]}$ a crepant resolution in codimension 2 of $S\times S/G$. Fujiki shows that $S(C_2)^{[2]}$ is a primitively symplectic orbifold.  
\end{ex}
We recall that the Nikulin orbifolds are defined in Definition \ref{Nikulin}.
\begin{prop}\label{FujikiM'}
The orbifold $S(C_2)^{[2]}$ is symplectic irreducible and of Nikulin type.
\end{prop}
\begin{proof}
Let $M'$ be the Nikulin orbifold obtained via the couple $(S^{[2]},i^{[2]})$, with $i^{[2]}$ the involution induced by $i$ on $S^{[2]}$.
It is enough to show that $S(C_2)^{[2]}$ is symplectic irreducible. Indeed $M'$ and $S(C_2)^{[2]}$ are both partial crepant resolutions in codimension 2 of $S\times S/G$ and so Corollary \ref{crepant} applies. 

Necessarily, $S(C_2)^{[2]}$ and $M'$ are bimeromorphic.
 Indeed, let $\Delta:= \left\{\left.\left\{x,x\right\}\in S\times S/G\ \right|\ x\in S\right\}$ and
 $\Sigma:=\left\{\left.\left\{x,i(x)\right\}\in S\times S/G\ \right|\ x\in S\right\}$. Let $U:=S\times
 S/G\smallsetminus (\Delta\cup \Sigma)$. The set $U$ can be seen as an open set in $S(C_2)^{[2]}$ and in
 $M'$. This gives rise to a natural bimeromorphism $f:S(C_2)^{[2]}\dashrightarrow M'$ (extending $\id_U$).

 By Lemma \ref{biracodim}, $f$ extends to an isomorphism from an open set $U'\subset S(C_2)^{[2]}$ to an
 open set $V'\subset M'$ such that $\codim S(C_2)^{[2]}\smallsetminus U'\geq 2$ and $\codim M'\smallsetminus
 V'\geq 2$ with $U\subset U'$ and $U\subset V'$. Moreover, $U$ contains all the singularities which are given by $\left\{\left.\left\{x_i,x_j\right\}\right|\ i<j\right\}$, where $(x_i)_{i\in\left\{1,...,8\right\}}$ are the fixed points of $i$ on $S$. It follows that $\pi_1(S(C_2)^{[2]}_{reg})=\pi_1(U'_{reg})=\pi_1(V'_{reg})=\pi_1(M'_{reg})=0$, where the index $reg$ means the regular subset. 
\end{proof}
\begin{rmk}
Actually, it can be deduced from the theory of wall divisors that $M'$ and $S(C_2)^{[2]}$ are isomorphic at least in some cases. 
If $S$ is general in the sense that $S$ does not contain any curve, then $M'$ and $S(C_2)^{[2]}$ are isomorphic according to Proposition \ref{isometry}, Corollary \ref{cor:desrK} and \cite[Section 4.1]{Ulrike2}.
\end{rmk}
\begin{rmk}
Note that this proposition is in contradiction with the Euler characteristic computed in \cite[Proposition 5.1, Corollary 5.7]{Marku-Tikho} and the one suggested in \cite[Remark 13.2 (4)]{Fujiki-1983}. However, we can prove that the Euler characteristic of $S(C_2)^{[2]}$ and $M'$ is actually 212 using two different methods. Indeed, it is proved in \cite[Proposition 2.40]{Menet-2015} using the link between the invariant cohomology and the cohomology of the quotient that $\chi(M')=212$. Moreover, in \cite[Proposition 3.6]{Fu-Menet}, using Blach--Riemann--Roch theorem \cite[Theorem 3.5 and 3.17]{Blache}, we provide an expression such that the Euler characteristic of a 4-dimensional primitively symplectic orbifold $X$ depends only on the second Betti number, the third Betti number and the singularities: 
$$b_4(X)+b_3(X)-10b_2(X)=46+s,$$
where $s$ only depends of the singularities.
That is:
\begin{equation}
\chi(X)=(48+s)+12b_2(X)-3b_3(X).
\label{Euler}
\end{equation}
More precisely $$s=\sum_{x\in\Sing X}s_x,$$
where $s_x$ only depends of the analytic type of the singularity $x$.
It was shown in \cite[(16)]{Fu-Menet}, that $s_x=-1$ if $x$ is a singularity analytically equivalent to $\C^4/-\id$.
By \cite[Propositions 2.8 and 2.40]{Menet-2015}, \cite[Corollary 5.7]{Marku-Tikho}, \cite[Theorem 13.1, Table 1, I.2]{Fujiki-1983}, $b_2(M')=b_2(S(C_2)^{[2]})=16$, $b_3(M')=b_3(S(C_2)^{[2]})=0$
and $M'$, $S(C_2)^{[2]}$ have 28 singularities analytically equivalent to $\C^4/-\id$.
Hence by (\ref{Euler}):
$$\chi(M')=\chi(S(C_2)^{[2]})=212.$$
\end{rmk}
\subsection{Application: example of non-separated orbifolds in their moduli space}\label{nonseparated}
In \cite{Marku-Tikho}, Markushevich and Tikhomirov provide a primitively symplectic orbifold related by a Mukai flop to a Nikulin orbifold (see Definition \ref{Nikulin}). 
The construction starts with a K3 surface $S$ endowed with an anti-symplectic involution $i$ such that the quotient $X/i$ is a del Pezzo surface of degree 2. We endow $S$ with a polarization $H$ which is a pull-back of the anti-canonical bundle of $X/i$. In particular, the divisor $H$ is very ample and provides an embedding of $S$ in $\mathbb{P}^3$ (see \cite[Lemma 1.1]{Marku-Tikho}).
Then we consider $\mathscr{M}=M_{S}^{H,ss}(0,H,-2)$, the moduli space of $H$-semi-stable sheaves on $S$ with Mukai vector $(0,H,-2)$.
Then the Markushevich--Tikhomirov orbifold $\mathcal{P}$ is constructed as a connected component of the fixed
locus of the involution $i^*\circ\tau$ where $\tau$ is a generalization of the dual map adapted to
$\mathscr{M}$ (see  \cite[Section 3]{Marku-Tikho} for more details).

Restricting to the case where $S\inj \bP^3$ does not contain any line, one can consider the Beauville involution $\iota$ on $S^{[2]}$ (see \cite[p.21]{Beauville1982}) which is an antisymplectic
involution. The composition $i^{[2]}\circ \iota$ is therefore symplectic. Hence, we can consider $M'$ the partial resolution in codimension 2 of $S^{[2]}/i^{[2]}\circ\iota$ (see Example \ref{exem}).

\begin{prop}\label{rappel}
The varieties $\mathcal{P}$ and $M'$ are irreducible symplectic orbifolds of dimension 4 with only 28 singular points analytically equivalent to $(\mathbb{C}^4 / \left\{\pm1\right\},$ $0)$. Moreover they are related by a Mukai flop. In particular the orbifold $\mathcal{P}$ is of Nikulin type.
\end{prop}
\begin{proof}
By Example \ref{exem} we know the properties of $M'$. By \cite[Theorem 3.4 and Corollary 5.7]{Marku-Tikho} we know that $\mathcal{P}$ is a primitively symplectic orbifold related to $M'$ by a Mukai flop. As explained in the proof of \cite[Lemma 5.3]{Marku-Tikho}, the indeterminacy locus of this Mukai flop does not contain any singularities of $\mathcal{P}$.
It follows that $\pi_1(\mathcal{P}_{reg})=\pi_1(M'_{reg})=0$. That is $\mathcal{P}$ is an irreducible symplectic orbifold.
We conclude with Proposition \ref{mainprop2} that $M'$ and $\mathcal{P}$ are deformation equivalent.
\end{proof}
Therefore, it is natural to ask whether $M'$ and $\mathcal{P}$ are isomorphic. We will answer to this question in what follows.

Assume that  $(S,i)$ is a very general K3 surface as in the previous paragraph, i.e.~satisfying the additional
condition $\Pic S=H^2(S,\Z)^i$.
Denote by $j:H^2(S,\Z)\rightarrow H^2(S^{[2]},\Z)$ the natural injection constructed in \cite[Proposition 6]{Beauville1983}.
We have the following description of the action of the Beauville involution.
\begin{prop}[\cite{OG-2005}, Proposition 4.1]\label{Ogradydy}
 The involution $\iota^*$ restricted to $\Pic S^{[2]}$ is the reflection in the span of $\theta:=j(H)-\delta$, 
 where $\delta$ is half the class of the diagonal.
\end{prop}
\begin{nota}
We denote by $\mathcal{M}$ the moduli space of marked irreducible symplectic orbifolds of Nikulin type.
\end{nota}
\begin{thm}\label{mainappli}
 Let $\mathcal{P}$ be a very general Markushevich--Tikhomirov variety (constructed from a very general 2-elementary K3 surface). 
 Let $\rho: M'\dashrightarrow \mathcal{P}$ be the Mukai flop of Proposition \ref{rappel}.
 Let $\varphi$ be a marking for $\mathcal{P}$, then $(\mathcal{P},\varphi)$ and $(M',\rho^*\circ\varphi)$ are distinct non-separated points in
 $\mathcal{M}$. More precisely, the orbifolds $M'$ and $\mathcal{P}$ are not isomorphic.
\end{thm}
\begin{proof}
Because of Propositions \ref{rappel} and \ref{mainprop2}, we know that $(\mathcal{P},\varphi)$ and $(M',\rho^*\circ\varphi)$ are non-separated points in $\mathcal{M}$.
Now, we are going to prove that $M'$ and $\mathcal{P}$ are not isomorphic. We will assume that there exists an
isomorphism $\mathcal{P} \simeq M'$  and we will find a contradiction.
The composition $\psi:M'\dashrightarrow \mathcal{P}\simeq M'$ induces a bimeromorphism of $M'$ to itself. 
We will show that this bimeromorphism is necessarily an isomorphism; since $\rho$ is not an isomorphism it will be a contradiction.

Let $\pi:S^{[2]}\rightarrow M:=S^{[2]}/i^{[2]}\circ \iota$ be the quotient map. Let $\Sigma$ be the surface fixed by the involution $i^{[2]}\circ\iota$ and $\Sigma'$ the exceptional divisor of the blow-up
$r:M'\rightarrow M$. We have a commutative diagram:
\begin{equation}
\xymatrix{M'\ar@{.>}[r]^{\psi}\ar[d]^r& M'\ar[d]^r\\
M \ar@{.>}[r]^{\psi_0}& M.}
\label{digra}
\end{equation}
Because of Proposition \ref{Ogradydy}:
\begin{equation}
\Pic M\otimes \Q=\Q\pi_*(\theta)\ \text{and}\ \Pic M'\otimes \Q=\Q r^*\pi_*(\theta)\oplus\Q\Sigma'.
\label{pica}
\end{equation}
Let $\mathcal{L}$ be the line bundle on $M$ associated to $\pi_*(\theta)$.
By Lemma \ref{biracodim}, $\psi$ is an isomorphism in codimension 1. It follows that $\psi_0$ is also an isomorphism in codimension 1.
Hence by \cite[Corollary II 3.15]{Banica-Stanasila} we have:
\begin{equation}
H^0(M,\mathcal{L}^n)=H^0(M,\psi_0^*(\mathcal{L})^n),
\label{Bonica}
\end{equation}
for all $n\geq0$.

By construction $M$ is projective and by (\ref{pica}) $\mathcal{L}$ is an ample line bundle.
It follows from (\ref{Bonica}), that: $$\psi_0^*(\pi_*(\theta))=\pi_*(\theta).$$
Therefore, by commutativity of the diagram (\ref{digra}), we have:
$$\psi^*(r^*(\pi_*(\theta)))=r^*(\pi_*(\theta)).$$

We recall the developed form of the Fujiki formula (see \cite[Theorem 3.17]{Menet-2020}):
\begin{equation}
\alpha_1\cdot\alpha_2\cdot\alpha_3\cdot\alpha_4=\frac{c_{M'}}{24}\sum_{s\in S_{4}}B_{M'}(\alpha_{s(1)},\alpha_{s(2)})
B_{M'}(\alpha_{s(3)},\alpha_{s(4)}),
\label{Fujikiform}
\end{equation}
with $B_{M'}$ the Beauville--Bogomolov form and $c_{M'}$ the Fujiki constant.
Hence: 
$$r^*(\pi_*(\theta))^3\cdot \widetilde{\Sigma}=c_{M'}B_{M'}(\pi_*(\theta),\pi_*(\theta))B_{M'}(\pi_*(\theta),\Sigma').$$
By the projection formula (\cite[Remark 2.8]{Menet-2020}):
$$r^*(\pi_*(\theta))^3\cdot \widetilde{\Sigma}=0.$$
Moreover by \cite[Proposition 7.10]{Menet-2018}:
\begin{equation}
B_{M'}(r^*(\pi_*(\theta)),r^*(\pi_*(\theta)))\neq0.
\label{BB2}
\end{equation}
It follows:
\begin{equation}
B_{M'}(r^*(\pi_*(\theta)),\Sigma')=0.
\label{BB1}
\end{equation}
Moreover, by Lemma \ref{isometry}, $\psi^*$ is an isometry on $(\Pic M',B_{M'})$. Therefore by (\ref{BB2}) and (\ref{BB1}), we have $\psi^*(\Sigma')=\pm\Sigma'$.
We conclude as before using Banica and Stanasila theorem that: 
$$\psi^*(\Sigma')=\Sigma'.$$
It follows that $\psi^*$ induces the identity on $\Pic M'$. So $\psi^*$ necessarily sends an ample class to an ample class. 
Hence $\psi$ extends to an isomorphism.
\end{proof}
\section{The Kähler cone}\label{Kähler}
\subsection{Generalization of Boucksom's criterion}


This section is devoted to the proof of the following theorem. We recall that the positive cone $\mathcal{C}_X$ is defined in Section \ref{posi}.
\begin{thm}\label{Kählercriterion}
Let $X$ be an irreducible symplectic orbifold. Let $\alpha\in \mathcal{C}_X$ such that $C\cdot \alpha>0$ for all rational curves $C\subset X$. Then 
$\alpha$ is a Kähler class.
\end{thm}
\begin{lemme}\label{intercone}
Let $X$ be an irreducible symplectic orbifold. Let $\alpha\in \overline{\mathcal{C}_X}$ such that $C\cdot \alpha\geq0$ for all rational curves $C\subset X$. Then 
$\alpha\in \overline{\mathcal{K}_X}$.
\end{lemme}
\begin{proof}
Following the same argument as \cite[Proposition 3.2]{Huybrechts2003}, this lemma is a consequence of Propositions \ref{mainprop} and \ref{oldprop}.
\end{proof}
We will adapt the proof of \cite{Boucksom-2001} to obtain Theorem \ref{Kählercriterion} from the previous lemma. 

Let $X$ be an irreducible symplectic orbifold endowed with a marking $\varphi\colon H^2(X,\bZ)\to \Lambda$ and
let $\alpha\in \mathcal{C}_X$ as in the statement of Theorem \ref{Kählercriterion}. Let
$f:\mathscr{X}\rightarrow \Def(X)$ be the Kuranishi deformation of $X$. We denote $o:=f(X)$. As explained in
Section \ref{per}, there exists $U$ an open neighborhood of $o\in \Def(X)$ such that the period map induces an
isomorphism $\mathscr{P}: U\rightarrow \mathscr{P}(U)$.
We denote:
 $$W_{\alpha}:=\Vect_{\C}(\varphi(\alpha), \varphi(\Rea \sigma_X), \varphi(\Ima \sigma_X))\ \text{and}\ T(\alpha)|_{U}:= \mathscr{P}^{-1}(\mathscr{P}(U)\cap \bP(W_{\alpha})).$$
By shrinking $U$ if necessary, we may assume that $U$ and $\mathscr{P}(U) \cap \bP(W_\alpha)$ are simply connected.

Then we consider the one dimensional deformation of $X$: 
\begin{equation}
f:f^{-1}(T(\alpha)|_U)\rightarrow T(\alpha)|_U
\label{localtwistor}
\end{equation}
For simplicity, we continue to denote $f^{-1}(T(\alpha)|_U)$ by $\mathscr{X}$.
The deformation $f:\mathscr{X}\rightarrow T(\alpha)|_U$ can be seen as a local twistor space.

For all $t\in T(\alpha)|_U$, we denote by $\varphi_t:=\varphi\circ u_t$ the markings, with
$u_t:H^*(\mathscr{X}_t,\Z)\rightarrow H^*(X,\Z)$ the parallel transport operator induced by the
deformation $f$. We denote by $\sigma_t$ the holomorphic 2-form on $\mathscr{X}_t$ and choose $\alpha_t$ such
that $\R_{>0}\alpha_t=\varphi_t^{-1}(W_{\alpha})\cap H^{1,1}(\cX_t,\bR)\subset \mathcal{C}_{\mathscr{X}_t}$.
\begin{lemme}\label{interproof}
Keeping the same notation as above, 
there exists $t\in T(\alpha)|_U$ such that $\alpha_t$ is a Kähler class on $\mathscr{X}_t$.
\end{lemme}
\begin{proof}
If there exists $t\in T(\alpha)|_U$ such that $\mathscr{X}_t$ does not contain any rational curve, by Lemma \ref{intercone}, we have $\mathcal{C}_{\mathscr{X}_t}=\mathcal{K}_{\mathscr{X}_t}$.
Therefore, we assume that for all $t\in T(\alpha)|_U$, the orbifold $\mathscr{X}_t$ contains a rational curve and we will find a contradiction.
Let $\beta\in H^{4n-2}(X,\Z)$, we consider the set $S_{\beta}$ of $t\in U$ such that $u_t^{-1}(\beta)$ is a cohomology class of type $(2n-1,2n-1)$. As explained in the beginning of \cite[Section 4.2]{Menet-2020}, the $S_{\beta}$ are analytic subsets of $U$.
By \cite[Proposition 2.12]{Menet-2020} the class of a rational curve is of type $(2n-1,2n-1)$, hence $T(\alpha)|_U\subset \cup_{\beta\in H^{4n-2}(X,\Z)} S_{\beta}$.
Since $H^{4n-2}(X,\Z)$ is countable, necessarily there necessarily exists $\beta\in H^{4n-2}(X,\Z)$ and a set $\mathscr{U}$, which is the complement of countably many points in $T(\alpha)|_U$, such that $u_t^{-1}(\beta)$ is the class of a rational curve for all $t\in \mathscr{U}$. 
Let $\beta$ be such a class and $\mathscr{U}$ such a set. Then by \cite[Theorem 1.2]{Nobile}, $\beta$ is the class of an effective combination of rational curves in $X$.
However, $\sigma_t\cdot u_t^{-1}(\beta)=0$ for all $t\in \mathscr{U}$, that is $u_{t}(\sigma_t)\cdot \beta=0$
for all $t\in \mathscr{U}$. This implies that $w\cdot \beta = 0$ for all $w\in W_\alpha$, and in particular $\alpha\cdot \beta=0$ which is a contradiction.
\end{proof}
\begin{proof}[Proof of Theorem \ref{Kählercriterion}]
Let $t\in T(\alpha)|_U$, as in Lemma \ref{interproof}. We consider:
$$W_t=\Vect_{\C}(\varphi_t(\alpha_t), \varphi_t(\Rea \sigma_t), \varphi_t(\Ima \sigma_t))\ \text{and}\ T_{W_t}:=\mathbb{P}(W_t)\cap \mathcal{D}.$$
By Lemma \ref{interproof} $\alpha_t$ is a Kähler class, and therefore we can apply Theorem \ref{Twistor} to consider $f':\mathscr{X}'\rightarrow T(\alpha_t)$ the twistor space associated to $\mathscr{X}_t$ and $\alpha_t$ (the period map provides an isomorphism $\mathscr{P}:T(\alpha_t)\rightarrow T_{W_t}$).

However, by construction $W_t=W_{\alpha}$. Hence, we have $\varphi(\sigma_o)\in T_{W_t}$. Let
$o_t\in T(\alpha_t)$ be such that $\cP(o_t)=\varphi(\sigma_o)\in T_{W_\alpha}$ and $X':=f'^{-1}(o_t)$. We denote by $\varphi'$ the
mark on $X'$ which arises by parallel transport from $\phi$. We consider $\mathcal{C}_{X'}\cap \varphi'^{-1}(W_t)=\R_{>0}\alpha'$. Since $f'$ is a twistor space, $\alpha'$ is a Kähler class on $X'$. Moreover by construction $\varphi(\alpha)=\varphi'(\alpha')$.
The two marked irreducible symplectic orbifolds $(X,\varphi)$ and $(X',\varphi')$ have the same image by the period map, hence 
by Theorem \ref{mainGTTO}, they are non-separated points in their moduli space.

So by \cite[Remark 3.24]{Menet-2020} there exists a $2n$-dimensional analytic cycle $\Gamma=Z+\sum Y_k\subset X\times X'$ such that:
\begin{itemize}
\item[(i)]
$Z$ induces a bimeromorphism:
$$\begin{array}{cccccc}
&&Z&&\\
&\pi\swarrow&&\searrow\pi'&\\
&X~-&\stackrel{g}{-}&\to~X';&
\end{array}$$
\item[(ii)]
the components $Y_k$ dominate neither $X$ nor $X'$;
\item[(iii)]
$[\Gamma]_*(\alpha)=\varphi'^{-1}\circ\varphi(\alpha)=\alpha'$, where
$[\Gamma]_*(\alpha):=\pi'_*([\Gamma]\cdot \pi^*\alpha)$.
\end{itemize}
Using Lemmas \ref{orbiornotorbi} and \ref{biracodim}, for the same reason as explained in \cite[proof of Theorem 2.5]{Huybrechts2003}, the maps $[Y_k]_*:H^2(X,\C)\rightarrow H^2(X',\C)$ are trivial.
Then by Proposition \ref{mainprop}, the bimeromorphism induced by $Z$ extends to an isomorphism.
\end{proof}



\subsection{Wall divisors} \label{sec:wall-divs}
In this section we will define wall divisors on irreducible symplectic orbifolds as Mongardi did in \cite[Definition 1.2]{Mongardi13}. 
The results of this section are obtained by adapting ideas from \cite{Mongardi13} and \cite{AmerikVerbitsky16X}.
\begin{defi}
  Let $X$ be an irreducible symplectic orbifold of dimension $2n$.
	\begin{itemize}
	\item[(i)]
We denote by $\BK_X$ the \emph{birational Kähler cone} which is the union $\cup f^{*}\mathcal{K}_{X'}$ for $f$
running through all bimeromorphic maps between $X$ and any irreducible symplectic orbifold $X'$. This is
well-defined by Lemma \ref{isometry}.
\item[(ii)]
We call the \emph{Mori cone} the cone in $H^{2n-1,2n-1}(X,\R)$ generated by the classes of irreducible curves contained in $X$.  
\item[(iii)]
 A \emph{Kähler-type chamber} of the positive cone $\mathcal{C}_X$ is a subset of the form $g\left[f^*(\mathcal{K}_{X'})\right]$, where $g\in Mon^2_{Hdg}(X)$, and $f:X\dashrightarrow X'$ is a bimeromorphic map to an irreducible symplectic orbifold $X'$.
\item[(iv)]
A \emph{Birational Kähler type chamber} of the positive cone $\mathcal{C}_X$ is a subset of the form $g(\BK_X)$, where $g\in Mon^2_{Hdg}(X)$.
\item[(v)]
A prime Cartier divisor $D$ with $q(D)<0$ is called a \emph{prime exceptional divisor}.
\end{itemize}
\end{defi}

\begin{defi}\label{walldefi}
Let $X$ be an irreducible symplectic orbifold and let $D\in\Pic(X)$. Then $D$ is called a \emph{wall divisor}
if $q(D)<0$ and $g(D^{\bot})\cap \BK_X =\emptyset$, for every parallel transport Hodge isometry $g \in \MonHdg(X)$.
\end{defi}
Note that an equivalent formulation for this definition is that $D^\perp$ does not intersect any Kähler-type
chamber.

One of our main results on wall divisor is Theorem \ref{thm:defo-for-WDintro}, of which we prove item (i) below.
\begin{thm}\label{thm:defo-for-WD}
Let $(X,\varphi)$ and $(Y,\psi)$ be two marked irreducible symplectic orbifolds in the same connected component of their moduli space. Let $D$ be a wall divisor of $X$ such that $(\psi^{-1}\circ \varphi)(D)\in \Pic(Y)$. Then $(\psi^{-1}\circ \varphi)(D)$ is a wall divisor on $Y$.
\end{thm}
\begin{proof}
We adapt the proof of \cite[Theorem 1.3]{Mongardi13}.
Let $\mathcal{M}_{\Lambda}^{\circ}$ be a connected component of the moduli spaces $\mathcal{M}_{\Lambda}$ which contains $(X,\varphi)$ and $(Y,\psi)$.
Let $D\in \Pic(X)$ with $q(D)<0$ such that $D_Y\coloneqq (\psi^{-1}\circ \varphi)(D)\in \Pic(Y)$.
Assume that $D_Y$ is not a wall divisor on $Y$. We need to deduce that $D$ cannot be a
wall divisor on $X$. Note that we may assume that $D_Y$ is primitive.

Since by assumption $D_Y$ is not a wall divisor on $Y$, there exists another marking $\psi_2$ on $Y$ such that
${\psi_2}^{-1}\circ \psi$ is a parallel transport Hodge isometry  which satisfies
$({\psi_2}^{-1}\circ \psi)(D_Y^\perp)\cap \BK_Y  =({\psi_2}^{-1}\circ \varphi)(D^\perp)\cap \BK_Y \neq \emptyset$.
Therefore, up to replacing $\psi$ by $\psi_2$, we can assume
$D_Y^\perp\cap \BK_Y \neq \emptyset$.

By the definition of the birational Kähler cone, this implies that there exists a bimeromorphic map $f\colon Y\dashrightarrow
Y'$ between irreducible symplectic orbifolds and such that $D_Y^\perp \cap f^*(K_{Y'})  \neq \emptyset$. However, by Proposition \ref{mainprop2} we know that
$(Y,\psi)$ and $(Y',\psi \circ f)$ are deformation equivalent. 
Thus, up to replacing $Y$ by $Y'$ and $D_Y$ by ${f^{-1}}^* D_Y$ we can assume that $D_Y^\perp \cap K_Y  \neq \emptyset$.

Choose a Kähler class $\omega\in D_Y^\perp \cap K_Y$ such that $\bZ D_Y$ are the only elements in
$H^{1,1}(Y,\bZ)$ which are orthogonal to $\omega$. This is possible, since $D_Y^\perp \cap K_Y$ $\subseteq D_Y^\perp$ is a non-empty open  subset.
Consider the twistor family $\cX^0 \to \bP^1$ associated to $\omega$ (see Theorem \ref{Twistor}). By construction of the
twistor space,  the period map
identifies this $\bP^1$ with the twistor line $T_W$ for the positive 3-space $W\coloneqq \Vect_\bR(\psi(\omega), \psi(\Rea\sigma_Y), \psi (\Ima\sigma_Y))$.
The choice of $\omega$ implies that  $W^\perp \cap \psi(H^{1,1}(Y,\bZ))=\Z\psi(D_Y)$. Therefore, a very
general fiber $(Z,\eta)$ of the twistor space satisfies that $\Pic(Z)=D_Z\cdot \bZ$, where
$D_Z\coloneqq \eta^{-1}\circ \psi(D_Y)$.
Note that $D_Z$ cannot be a wall divisor on $Z$, since by the construction of the twistor family $Z$ is
equipped with a Kähler form $\omega_Z \in W \subseteq D_Z^\perp$.

We claim that this implies $\mathcal{K}_Z=\mathcal{C}_Z$. Indeed, $Z$ cannot contain an effective curve, since
$H^{2n-1,2n-1}(Z,\bZ)$ consists exactly of the ray  dual to $D_Z$, which pairs trivially with the Kähler
class $\omega_Z$.  Theorem \ref{Kählercriterion} therefore
implies that $\dK_Z=\dC_Z$ as claimed.

Note that this has the following consequence. Let $W'\subseteq \Lambda_{\bR}$ be any positive 3-space such
that $\mathscr{P}(Z,\eta)\in T_{W'}$ (i.e.~$W'\supseteq \Vect_\bR( \Rea (\sigma_Z), \Ima (\sigma_Z)$).
Then there exists a twistor family $\cX_1 \to \bP^1$ such that $\bP^1$ is identified with $T_{W'}$ via the
period map. Indeed, $W'\cap \Vect_\bR( \Rea (\sigma_Z), \Ima (\sigma_Z))^\perp$ is a one-dimensional
positive subspace, which therefore contains an element $\omega_{W'}\in \dC_Z =\dK_Z$ and $\cX_1$ is obtained
as the twistor family associated to $\omega_{W'}$. 

We now connect $\mathscr{P}(Z,\eta)$ back to $\mathscr{P}(X,\phi)$ via convenient twistor spaces.
Let $N$ be the sublattice of $\Lambda$ defined by $N:=\varphi(D)^{\bot}=\psi(D_Y)^\perp =\eta(D_Z)^\perp$. Since $q(D)<0$, the lattice $N$ has signature $(3,\rk N-3)$.
Therefore by \cite[Proposition 3.7]{Huybrechts12}  the period domain
\begin{equation*}
  \mathcal{D}_N:=\mathbb{P}\left(\left\{\sigma\in
  N\otimes\mathbb{C}\ |\ \sigma^{2}=0,\ (\sigma+\overline{\sigma})^{2}>0\right\}\right)\subseteq
  \mathbb{P}(N\otimes\mathbb{C})\subseteq \bP(\Lambda\otimes \bC)
\end{equation*}
is connected by generic twistor lines. Note that the proof of \cite[Proposition 3.7]{Huybrechts12} in fact shows
that the twistor lines can be chosen in a such a way that they intersect at very general points of $\dD_N$.
In particular, we can connect  $\mathscr{P}(Z,\eta)$ and $\mathscr{P}(X,\phi)$ by
generic twistor lines in $\dD_N$ intersecting in very general points.

As we have seen above, any generic twistor line $T_{W'}$ through $\mathscr{P}(Z, \eta)$ lifts to a twistor family. Repeating the
same arguments as above, we observe that  the fiber $(Z',\eta')$ over a very general point of this family satisfies $\Pic(Z')=\Z D_{Z'}$, with $D_{Z'}\coloneqq \eta'^{-1}\phi(D)$ which cannot be wall
divisor, since it is orthogonal to a Kähler class. As we observed above, this implies that
$\dC_{Z'}=\dK_{Z'}$, which allows us to repeat the same arguments.

In this way, we obtain an irreducible symplectic orbifold $(X',\varphi')$ with $\mathscr{P}(X',\varphi')=\mathscr{P}(X,\varphi)$,
which satisfies $ (\varphi'^{-1} \circ \varphi (D))^\perp \cap \dK_{X'}  \neq \emptyset$ (since it comes with a
Kähler class orthogonal to $\varphi'^{-1} \circ \varphi (D)$ by construction).
Since we assumed that $(X,\varphi)$ and $(Y,\psi)$ are deformation equivalent, $\psi^{-1} \circ \varphi$ is a
parallel transport operator. On the other hand $\varphi'^{-1} \circ \psi$ is a parallel transport operator by
construction. Therefore $\varphi'^{-1}\circ \varphi$ is a parallel transport operator, which is a Hodge
isometry since $\mathscr{P}(X',\varphi')=\mathscr{P}(X,\varphi)$.
By \cite[Theorem 1.1 and Proposition 3.22]{Menet-2020}, there is a bimeromorphism $f:X\dashrightarrow X'$.
Proposition \ref{mainprop2} shows that $f^*$ (and thus $f^*\circ \varphi'^{-1}\circ \varphi$) is a parallel transport Hodge isometry.
By construction, we have $\emptyset \neq (f^{*}\circ\varphi'^{-1}\circ \varphi(D))^{\bot}\cap
f^*(K_{X'})\subseteq (f^{*}\circ\varphi'^{-1}\circ \varphi(D))^{\bot}\cap \BK_X$. Thus $D$ is not a
wall divisor, which is what we wanted to show.
\end{proof}

\begin{defi}\label{def:setW}
Let $\dM_\Lambda$ be the moduli space of marked irreducible symplectic orbifolds with Beauville--Bogomolov lattice $\Lambda$ (see Section \ref{per}).

  For a given irreducible symplectic orbifold, we denote by $\cW_X\subseteq \Pic(X)$ the set of all primitive wall divisors.
  Furthermore, after choosing a connected component $\dM_\Lambda^\circ$, let $\cW_\Lambda \subseteq \Lambda$ be the set of all classes $\alpha\in \Lambda$ such that
  $\psi^{-1}(\alpha)\in \cW_Y$ is a primitive wall divisor for some $(Y,\psi)\in \dM^\circ_\Lambda$.
\end{defi}

As an immediate consequence of Theorem \ref{thm:defo-for-WD} we obtain the following corollary.
\begin{cor}\label{wall}
Let $\Lambda$ be a lattice of signature $(3,\rk \Lambda -3)$ and $\dM_{\Lambda}^{\circ}$ a connected component
of the associated moduli space of marked irreducible symplectic orbifolds. Then for any $(X,\varphi)\in
\dM_{\Lambda}^{\circ}$ the set $\varphi^{-1}(\cW_\Lambda)\cap \Pic(X)$ consists of the (primitive) wall
divisors on $X$.
\end{cor}
\begin{ex}[\cite{Mongardi13}, Proposition 2.12]
If $\mathscr{M}_{\Lambda}^{\circ}$ is a connected component of the moduli space of marked K3 surface, then:
$$\cW_\Lambda=\left\{\left.D\in \Lambda\ \right|\ D^2=-2\right\}.$$
If $\mathscr{M}_{\Lambda}^{\circ}$ is a connected component of the moduli space of marked irreducible
symplectic manifolds  deformation equivalent to a Hilbert scheme of 2 points on a K3 surface, then: $$\cW_\Lambda=\left\{\left.D\in \Lambda\ \right|\ q(D)=-2\right\}\cup\left\{\left.D\in \Lambda\ \right|\ q(D)=-10\ \text{and}\ D\cdot\Lambda\subset2\Z \right\}.$$
\end{ex}

\begin{prop}\label{prop:boundedqW}
Let $\Lambda$ be a lattice of signature $(3,\rk \Lambda -3)$. We assume that $\rk \Lambda \geq 5$.
	Then there exists a natural number $N\in \bN$ such that for every primitive wall divisor $D\in \cW_\Lambda$ the
  Beauville--Bogomolov square of $D$ satisfies $-N< q(D)<0$.
\end{prop}
\begin{proof}
The proof is an adaptation of \cite{AmerikVerbitsky16X}.
Let $\mathcal{M}_{\Lambda}^o$ be the connected component associated to $\mathscr{W}_{\Lambda}$. Let $(X,\varphi)\in \mathcal{M}_{\Lambda}^o$.
We set $\Lambda_{\K}:=\Lambda\otimes\K$ for any field $\K$. 
Let $\Gr_+$ be the Grassmannian of all positive three-space in $\Lambda_{\R}$.
Let $$\mathscr{W}_{\Lambda}^{\bot}:=\left\{\left.W\in \Gr_+\right|\ W^{\bot}\cap \mathscr{W}_{\Lambda} \neq\emptyset\right\}.$$
By \cite[Theorem 8.2]{Bakker-Lehn-GlobalTorelli}, $\Mon^2(X)$ is of finite index in $O^+(\Lambda)\subset O(\Lambda)$; therefore it is an arithmetic subgroup of $O^+(\Lambda)$.
Hence, we are in the framework of \cite[Theorem 1.2]{AmerikVerbitsky16X}; the theorem says that if $\mathscr{W}_{\Lambda}^{\bot}$ is not dense in $\Gr_+$ then the action of $\Mon^2(X)$ on $\mathscr{W}_{\Lambda}$ has a finite number of orbits. In this case, since the elements of $\Mon^2(X)$ preserve the Beauville--Bogomolov square, the set $\{q(D) |
  D\in \mathscr{W}_{\Lambda}\}$ will be finite, which will prove the claim.
Hence, we only have to prove that $\mathscr{W}_{\Lambda}^{\bot}$ is not dense in $\Gr_+$.

Let $U$ be an open subset of the Kuranishi space $\Def(X)$, such that the period map $\mathscr{P}:U\rightarrow\mathcal{D}_{\Lambda}$ is an isomorphism onto its image (see Section \ref{per}), and denote by $f$ the associated family $\mathscr{X}\rightarrow U$.
Let $\mathcal{K}_U$ be the subspace of $R^2f_*\R\rightarrow U$ which contains the Kähler classes for each fiber $\mathscr{X}_s$.
We denote by $u_s:H^2(\mathscr{X}_s,\R)\rightarrow H^2(X,\R)$ the parallel transport operator induced by $f$. 
The period map can be extended to $\mathcal{K}_U$ in a natural way. For this purpose, we consider the following period domain:
$$\Omega:=\left\{\left.(x,\omega)\in\mathbb{P}(\Lambda_{\C})\times \Lambda_{\R}\right|\ x^2=0,\ x\cdot \overline{x}>0,\ x\cdot \omega=0\right\}.$$
Then the period map is given by:
$$\mathscr{P}_{U}(\omega)=\left(\varphi\circ u_s(H^{2,0}(\mathscr{X}_s)),\varphi\circ u_s(\omega)\right),$$
with $s=g(\omega)$ for $g:\mathcal{K}_U\rightarrow U$.
According to the local Torelli theorem (see Section \ref{per}), we know that $\mathscr{P}_{U}$ is an injective differential map; moreover $\mathcal{K}_U$ and $\Omega$ have the same real dimension. So a dense set in $\Omega$ will intersect $\mathscr{P}_{U}(\mathcal{K}_U)$.
Finally, we can consider the following natural map:
$$\xymatrix@R0pt{\Phi:\ \Omega\ar[r]& \Gr_+\\
\ \ \ \ \ \ \ \ (x,\omega)\ar[r]&\Vect_{\R}(\Rea x,\Ima x, \omega).}$$
By definition \ref{walldefi}, $\Phi^{-1}(\mathscr{W}_{\Lambda}^{\bot})$ is not dense in $\Omega$. Therefore $\mathscr{W}_{\Lambda}^{\bot}$ cannot be dense in $\Gr_+$.
\end{proof}

\begin{lemme}\label{lem:locfin}
  Let $X$ be an irreducible symplectic orbifold such that $b_2(X)\geq5$.
  Then the chamber structure in $\dC_X$ cut out by wall divisors is locally finite in the following sense: Let
  $\Pi \subset \overline{\dC_X}$  be a rational polyhedral cone. Then the set
  $\{D\in  \cW_X \,|\, D^\perp \cap \Pi \neq \emptyset  \}$  is finite. 
\end{lemme}

\begin{proof}
  Using the boundedness of the Beauville--Bogomolov squares from Proposition \ref{prop:boundedqW}, this is an
  immediate consequence of \cite[Proposition 2.2]{MarkmanYoshioka14}. 
\end{proof}
\begin{rmk}
Proposition \ref{prop:boundedqW} and Lemma \ref{lem:locfin} will not be needed in the case $b_2(X)=4$ in order to prove the results of this section; another argument will be used in the case $b_2(X)=4$. We provide afterward some discussion when $b_2(X)=4$ in Section \ref{Kawa}.
\end{rmk}

We recall that the dual $\beta^{\vee}$ of a class $\beta\in H^{2n-1,2n-1}(X,\Q)$ is defined in Remark \ref{dualclass}.
\begin{prop}\label{extremalray}
Let $X$ be an irreducible symplectic orbifold. Let $R$ be an extremal ray of the Mori cone of $X$ of negative self intersection. Then any $D\in \Q R^{\vee}$ is a wall divisor.
\end{prop}
\begin{proof}
  Let $D\in \Q R^{\vee}$ primitive of negative square with $R$ an extremal ray. We claim that this implies that $D^\perp$
  cuts out a wall of $\mathcal{K}_X$, i.e.~$D^\perp \cap \del \mathcal{K}_X \subseteq D^\perp$ contains an open subset $V$ of $D^\perp$.  

  Assuming this, choose $\alpha \in V$ very general (i.e.~such that the only elements in $H^{1,1}(X,\bZ)$ which are orthogonal to
  $\alpha$ are collinear to $D$).
  Choose a marking $\phi\colon H^2(X,\bZ) \to \Lambda$ on $X$.
  Consider the 3-space $W\coloneqq \Vect_\bR( \phi(\alpha), \phi(\Rea \sigma), \phi(\Ima \sigma))$. By
  choice of $\alpha$ we have ${W}^\perp\cap \varphi(H^2(X,\bZ))=\Z\varphi(D)$. Therefore, $W$ gives rise to a generic twistor
  line $T_W$ in $\dD_N$, where $N\coloneqq \phi(D)^\perp$ as before. A priori $T_W$ is not coming from a twistor family, since $\alpha$ is not a
  Kähler class.

  However, as in (\ref{localtwistor}), locally around $(X,\phi)$, we can consider a family over some open 
  simply by using the local Torelli theorem (\cite[Theorem 3.17]{Menet-2020}). Let $(Y,
  \psi)$ be a very general element of this family (i.e.~an element with $\Pic(Y)=\Z\psi^{-1}\circ\phi(D)$.)

  Suppose, for the sake of contradiction, that $D$ is not a wall divisor. Then use Theorem \ref{thm:defo-for-WD} to
  deduce that $\psi^{-1}\circ\phi(D)$ is not a wall divisor either. Therefore, as in the proof of Theorem
  \ref{thm:defo-for-WD}, $\mathcal{C}_Y=\mathcal{K}_Y$.
  Note that $W\cap \psi(\Rea \sigma_Y)^\perp\cap \psi(\Ima \sigma_Y)^\perp \subseteq \psi(H^{1,1}(Y,\bR))$ is a
  positive 1-space and thus contains a Kähler form $\omega_Y$.
  By construction $\Vect_\bR( \psi(\omega_Y), \psi (\Rea \sigma_Y), \psi(\Ima \sigma_Y))=W$.

  Therefore, the twistor family $\cY$ associated to $\omega_Y$ surjects to the twistor line $T_W$
  associated to $W$. 
  In particular the fiber $(\cY_0, \psi_0)$ over $\mathscr{P}(X,\phi)$ is an irreducible symplectic orbifold, with
  Kähler class $\omega_{\cY_0}$, which satisfies
  \begin{equation*}
    \psi_0(\omega_{\cY_0})\in
    W\cap \psi_0(\Rea\sigma_{\cY_0})^\perp \cap \psi_0(\Ima\sigma_{\cY_0})^\perp
    =    W\cap \phi(\Rea\sigma)^\perp \cap \phi(\Ima\sigma)^\perp
    = \phi(\alpha) \cdot \bR \subseteq \phi(\del \mathcal{K}_X),
  \end{equation*}
  which is absurd.

  It remains to prove the claim. Since the Kähler cone is cut out by hyperplanes orthogonal to effective curves (see e.g.~Theorem \ref{Kählercriterion}) it
  suffices to prove that the extremal ray $R$ cannot be a limit of other extremal rays.

  Let us first prove the claim under the condition that $b_2(X)\geq 5$. Suppose, for the sake of contradiction, that $R$ is the
  limit of extremal rays $R_i$ for $i\to \infty$. Observe that, since all $R_i$ are extremal rays, there exist
  elements $0\neq \alpha_i \in \del \mathcal{K}_X \cap R_i^\perp$. We can choose these $\alpha_i$ such that they converge
  toward an element $0 \neq \alpha \in \del \mathcal{K}_X \cap R^\perp$. We will show that for each $\alpha_i$, there
  exists a wall divisor $D_i\in \cW_X$ with $D_i\perp\alpha_i$. This contradicts the locally finite structure of
  the walls from Lemma \ref{lem:locfin}. 
  Suppose, for the sake of contradiction, that there exists $i$ such that $\alpha_i$
  is not orthogonal to any wall divisor.
  Consider $W \coloneqq \Vect_\bR(\phi(\alpha_i),\phi(\Rea\sigma),\phi(\Ima \sigma))$ and repeat the first part
  of this proof to reach a contradiction (this is applicable since by construction $\phi^{-1}(W^\perp)$ does not contain
  any wall divisor).

    It remains to prove the claim when $b_2(X)=4$. If the Picard rank of $X$ is 0 or 1, there is nothing to prove. If $\rk \Pic X=2$, then the Mori cone has two extremal rays and the claim is automatically verified.
\end{proof}
 
\begin{defi}Given an irreducible symplectic orbifold $X$ endowed with
  a Kähler class $\omega$.
    Define $\cW_X^+\coloneqq \{D\in \cW_X\,|\, (D,\omega)_q>0\}$, i.e.~for every wall divisor, we choose the
  primitive representative intersecting (with respect to the BB form) positively any Kähler class.
\end{defi}
The following corollary gives another characterization of Kähler classes on irreducible symplectic orbifolds.
\begin{cor}\label{criterionwall}\label{cor:desrK}
  Let $X$ be an irreducible symplectic orbifold.
  Then
  $$
  \dK_X=\{\alpha\in \dC_X \,|\, (\alpha, D)_q>0\ \forall D\in \cW_X^+\}.$$
\end{cor}
\begin{proof}
Let $\alpha\in \dC_X$ satisfy $(\alpha, D)_q>0$ for all $D\in \cW_X^+$. We need to show that $\alpha\in
\dK_X$. The other inclusion follows immediately from the definitions.
By Theorem \ref{Kählercriterion}, it is enough to show that $\alpha\cdot R>0$ for every extremal rays of the Mori cone of $X$. 

If $q(R^{\vee})<0$ this is true by Proposition \ref{extremalray}.

If on the other hand $q(R^{\vee})\geq 0$, then $R^\vee\in \overline{\dC_X}$ (since effective curves pair
positively with Kähler classes), and therefore $\alpha \in \dC_X$ implies that $0<(R^\vee,\alpha)_q=R\cdot
\alpha$.
%
\end{proof}
The following proposition will be useful in Sections \ref{Kawa} and \ref{mir}.
\begin{prop}\label{Kählertype}
Every general class $\alpha\in \mathcal{C}_X$ belongs to some Kähler-type chamber.
\end{prop}
\begin{proof}
We adapt the proof of \cite[Lemma 5.1]{Markman11}.
By Proposition \ref{oldprop} and \cite[Remark 3.24]{Menet-2020}, there exists $X'$ an irreducible symplectic orbifold and a cycle $\Gamma:=Z+\sum_iY_i$ in $X\times X'$ such that $Z$ defines a bimeromorphism under the following commutative diagram:
$$\begin{array}{cccccc}
&&Z&&\\
&\pi'\swarrow&&\searrow\pi&\\
&X'~-&\stackrel{f}{-}&\to~X,&
\end{array}$$
where $\pi:X\times X'\rightarrow X$ and $\pi':X\times X'\rightarrow X'$ are the projection.
Moreover the map $\left[\Gamma\right]_*:H^2(X,\Z)\rightarrow H^2(X',\Z)$ defined for all $\beta\in H^2(X',\Z)$ by $\left[\Gamma\right]_*(\beta)=\pi'_*(\left[\Gamma\right]\cdot\pi^*(\beta))$ is a parallel transport operator and $\left[\Gamma\right]_*(\alpha)$ is a Kähler class of $X'$. We set $g:=f^*\circ\left[\Gamma\right]_*$. By Proposition \ref{mainprop2}, $f^*\in \Mon^2_{Hdg}(X)$ and so $g\in \Mon^2_{Hdg}(X)$. Hence $g^{-1}(f^*(\mathcal{K}_{X'}))$ is a Kähler-type chamber. Since $\alpha\in g^{-1}(f^*(\mathcal{K}_{X'}))$, this conclude the proof.
\end{proof}
Using the results on wall divisors, we can derive the following more explicit formulation of Proposition \ref{Kählertype}.
\begin{cor}\label{cor:alphamovedtoK}
  Let $X$ be an irreducible symplectic orbifold and let
  $\alpha\in \dC_X$ be such that $(\alpha,D)_q\neq 0$ for all $D\in \cW_X$.
  Then $\alpha$ belongs to a Kähler-type chamber.
\end{cor}

\begin{proof}
  First observe that there exists an open neighborhood $U$ of $\alpha$ such that $U\cap D^\perp =\emptyset$
  for all wall divisors $D\in \cW_X$. Suppose, for the sake of contradiction, that $\alpha$ is the limit of a sequence $(\alpha_i)_i$, with 
  $\alpha_i\in D_i^\perp$ for some wall divisor $D_i$, for any $i$. Then one can find a subsequence of these hyperplanes, which
  is converging, resulting in a hyperplane $W\ni \alpha$. This implies that $W$  does not intersect any Kähler
  chamber, since the $D_i^\perp$ do not. Therefore, $W$ is by definition the orthogonal complement of some
  wall divisor $D$. However, this gives the desired contradiction to the choice of $\alpha$.

  Applying Proposition \ref{Kählertype}, we deduce that there exists $\alpha'\in U$ such that $\alpha'$ is in a Kähler
  type chamber, i.e. there exists $g \in \MonHdg(X)$ and $f\colon X\dashrightarrow X'$ such that
  $\alpha'\in g^{-1}f^*(\mathcal{K}_Y)$.
  We consider the set
  $\cW_\alpha^+ \coloneqq  \{D\in \cW_X \,|\, (\alpha,D)_q >0\}$  of wall divisors pairing positively with
  $\alpha$. 
  Note that for every wall divisor $D\in \cW_X$ either $D$ or $-D$ is in this set (since we assumed that
  $(\alpha,D)_q \neq 0$).
  Furthermore, by the choice of $\alpha'\in U$, the conditions $(\alpha,D)_q>0$ and
  $(\alpha',D)_q>0$ are equivalent for every $D\in \cW_X$.
  Therefore,  $g^{-1}f^*(\mathcal{K}_Y)= \{\beta \in \dC_X\,|\, (\beta,D)_q>0\ \forall D\in \cW_\alpha^+\}$ by Corollary \ref{cor:desrK}.
  Since this set obviously contains $\alpha$, this concludes the proof.
\end{proof}
%
The proposition below is a characterization of the birational Kähler
cone of an irreducible symplectic orbifold. 
\begin{prop}\label{bira}
Let $X$ be an irreducible symplectic orbifold. Then $\alpha \in H^{1,1}(X,\R)$ belongs to the closure $\overline{\BK}_X$ of the birational Kähler cone $\BK_X$ if and only if $\alpha \in \overline{\mathcal{C}}_X$ and $(\alpha,[D])_q \geq 0$ for every prime exceptional divisor $D \subset X$.
\end{prop}

\begin{proof}
 We claim that a class $\alpha \in H^{1,1}(X,\R)$ lies in the closure $\overline{\BK}_X$ of the birational Kähler cone $\BK_X$ if and only if $\alpha \in \overline{\mathcal{C}}_X$ and $(\alpha,[D])_q \geq 0$ for all uniruled divisors $D \subset X$.The proof of this claim follows \cite[Proposition 4.2]{Huybrechts2003}, and make use of Lemma \ref{isometry} and Proposition \ref{oldprop}.

However, if $D'$ is a uniruled prime divisor with $q(D') \geq 0$, then $(\alpha,[D'])_q \geq 0$ holds automatically for all $\alpha \in \overline{\mathcal{C}}_X$. Thus, if $\alpha \in \overline{\mathcal{C}}_X$ and $(\alpha,[D])_q \geq 0$ for all prime exceptional divisors $D \subset X$, it follows that $\alpha \in \overline{\BK}_X$.

Conversely, suppose $\alpha \in \overline{\BK}_X$ and let $D$ be a prime exceptional divisor. Assume for contradiction that $(\alpha,[D])_q < 0$. Then there exists $\beta \in \BK_X$ such that $(\beta,[D])_q < 0$, contradicting Lemma \ref{biracodim}. So $(\alpha,[D])_q \geq 0$.
\end{proof}

\subsection{Some results when $b_2(X)=4$}\label{Kawa}

Let $X$ be an irreducible symplectic orbifold. In this section we provide some results that are missing in the case $b_2(X)=4$. 
Note that an example of irreducible symplectic orbifold with $b_2(X)=4$ have been provided in \cite[Theorem 1.11]{Lol6}.

For an element $\alpha\in H^2(X,\Z)$, we set the reflection $R_{\alpha}(x)=x-2\frac{(\alpha,x)_q}{q(\alpha)}\alpha$.
We also set:
$$\Mon^2_{Bir}(X):=\left\{\left.g^*\right|\ g:X\dashrightarrow X\ \text{is a bimeromorphism}\right\}.$$


\begin{lemme}\label{simplechambers}
Let $X$ be an irreducible symplectic orbifold with $b_2(X)=4$. We assume that $\mathcal{C}_X$ contains finitely many birational Kähler-type chambers. Then:
\begin{itemize}
\item[(i)] $\mathcal{C}_X$ contains at most two birational Kähler-type chambers;
\item[(ii)] Let $\rho: \Mon^2_{Hdg}(X) \rightarrow O(\NS(X))$ be the restriction morphism. If $\mathcal{C}_X$ contains two birational Kähler-type chambers, then there exists $B \in \cW_X^+$ such that $\rho(\Mon^2_{Hdg}(X)) = \left\langle \rho(R_{B}) \right\rangle$; otherwise, $\Mon^2_{Hdg}(X) = \Mon^2_{Bir}(X)$.
\end{itemize}
\end{lemme}

\begin{proof}
We may assume without loss of generality that $\operatorname{rk}\NS X = 2$. Let $\ell_1, \ell_2 \in H^{1,1}(X,\R)$ satisfy $q(\ell_1) = q(\ell_2) = 0$, such that $\R_{\geq 0} \ell_1$ and $\R_{\geq 0} \ell_2$ are the extremal rays of $\mathcal{C}_X$.

By Corollary \ref{cor:alphamovedtoK}, there exists a birational Kähler-type chamber with $\R_{\geq 0} \ell_1$ as an extremal ray. Therefore, every such chamber must have one of its extremal rays generated by an element of square zero. Hence, all these chambers must share either $\R_{\geq 0} \ell_1$ or $\R_{\geq 0} \ell_2$ as an extremal ray, so there are at most two birational Kähler-type chambers.

Now suppose $\mathcal{C}_X$ contains two birational Kähler-type chambers. Then $\overline{\BK}_X$ must have an extremal ray orthogonal to some $B \in \cW_X^+$. We aim to show that $\rho(\Mon^2_{Hdg}(X)) = \left\langle \rho(R_B) \right\rangle$.
Let $g \in \rho(\Mon^2_{Hdg}(X))$. Then necessarily $g(B) = \pm B$.
If $g(B) = B$, then $g(\ell_1) = \lambda \ell_1$ for some $\lambda \in \R^*_+$, since $g$ is an orientation-preserving isometry. Comparing $(B, \ell_1)_q = (g(B), g(\ell_1))_q$ yields $\lambda = 1$, hence $g = \id$.
If $g(B) = -B$, then $g(\ell_1) = \lambda \ell_2$ with $\lambda \in \R^*_+$ (again by orientation preservation). Similarly, $R_B(\ell_1) = \mu \ell_2$ for some $\mu \in \R^*_+$. By comparing scalar products as before, we get $\mu = \lambda$ and $g = R_{B|\NS(X)}$.

Finally, if $\mathcal{C}_X$ contains only one birational Kähler-type chamber, then $\mathcal{C}_X = \overline{\BK}_X$, and Theorem \ref{main}(ii) implies that $\Mon^2_{Hdg}(X) = \Mon^2_{Bir}(X)$.
\end{proof}
\begin{lemme}\label{infiteray}
Let $X$ be an irreducible symplectic orbifold with $b_2(X)=4$ and $\rk\NS(X)=2$. Assume that $\mathcal{C}_X$ contains infinitely many birational Kähler-type chambers. Then:
\begin{itemize}
  \item[(i)]  The two extremal rays of $\overline{\BK}_X$ are given by rays lying respectively in $D_1^\perp$ and $D_2^\perp$, for two prime exceptional divisors $D_1,D_2\subset X$.
  \item[(ii)] The birational Kähler-type chambers accumulate at the extremal rays of $\mathcal{C}_X$.
  \item[(iii)] Both extremal rays of $\mathcal{C}_X$ are irrational.
\end{itemize}
\end{lemme}

\begin{proof}
First, by \cite[Theorem 1.2]{Menet-2020}, $\NS(X)$ contains a class of positive Beauville--Bogomolov square, so $X$ is projective.

We claim there exists $N>0$ such that $q(D)>-N$ for every prime exceptional divisor $D$. Otherwise we could find prime exceptional divisors $D_n$ with $q(D_n)\to -\infty$. By \cite[Theorem 3.10]{Lehn2}, each reflection $R_{D_n}\in\Mon^2_{Hdg}(X)$ is an integral isometry, forcing the divisibility $\div(D_n)\to\infty$, which is impossible.

None of the extremal rays of $\overline{\BK}_X$ can coincide with those of $\mathcal{C}_X$, for otherwise the proof of Lemma \ref{simplechambers} would imply there are at most two birational Kähler-type chambers. Moreover, by \cite[Proposition 2.2]{MarkmanYoshioka14} and the preceding claim, no hyperplane $D^\perp$ (with $D$ a prime exceptional divisor) can accumulate along the boundary of $\overline{\BK}_X$. Hence, Proposition \ref{bira} shows that the two extremal rays of $\overline{\BK}_X$ are precisely two rays contained in $D_1^\perp$ and $D_2^\perp$, where $D_1$ and $D_2$ are distinct prime exceptional divisors. This proves (i).

Since $\mathcal{C}_X$ contains infinitely many birational Kähler-type chambers, there are infinitely many hyperplanes $g(D_i)^\perp$ (for $i\in\{1,2\}$ and $g\in\Mon^2_{Hdg}(X)$) cutting through $\mathcal{C}_X$. Again by \cite[Proposition 2.2]{MarkmanYoshioka14}, these hyperplanes can only accumulate at an irrational extremal ray of the positive cone. It follows that both extremal rays of $\mathcal{C}_X$ are irrational, because if only one extremal ray was irrational $\rho(\Mon^2_{Hdg}(X))$ would be trivial, contradicting the existence of infinitely many birational Kähler-type chambers. Furthermore, the same result ensures that the hyperplanes $g(D_i)^\perp$ accumulate exactly at those extremal rays, and hence the birational Kähler-type chambers accumulate there as well. This establishes (ii) and (iii).
\end{proof}
In \cite{Lehn2}, Lehn, Mongardi and Pacienza prove the Morrison--Kawamata cone conjecture for $X$ a $\Q$-factorial projective primitive symplectic variety with terminal singularities and $b_2(X)\geq5$. In this section, we propose a proof of the birational version of the Morrison--Kawamata cone conjecture for $X$ a projective irreducible symplectic orbifold with $b_2(X)=4$ using our previous results.
\begin{prop}\label{finiteindex}
Let $X$ be an irreducible symplectic orbifold. Then $\Mon^2(X)$ is of finite index in $O^+(H^2(X,\Z))$.
\end{prop}
\begin{proof}
This result has been proved in \cite[Theorem 1.1 (2)]{Bakker-Lehn-GlobalTorelli} when $b_2(X)\geq5$. So it remains to show the result for $b_2(X)=4$. 
We set $\Lambda:= H^2(X,\Z)$ and: $$\mathcal{M}_X:=\left\{\left.(Y,\psi)\in\mathcal{M}_{\Lambda}\right|\ Y\ \text{is deformation equivalent to}\ X\right\}.$$ We denote by $\overline{\mathcal{M}}_X$ the Hausdorff reduction of $\mathcal{M}_X$ (see Section \ref{GTT}). We set $\mathscr{P}_X: \overline{\mathcal{M}}_X\rightarrow \mathcal{D}_{\Lambda}$ the period map.
Let $p\in\mathcal{D}_{\Lambda}$ such that $p^{\bot}\cap \Lambda=\Z \alpha$ with $\alpha^2>0$.
Step 1 of the proof of Theorem 8.2 in \cite{Bakker-Lehn-GlobalTorelli} does not require that $b_2(X)\geq5$; so the statement is also true in our case and we have that
$\mathscr{P}_X^{-1}(p)$ is finite. Therefore by Theorem \ref{mainGTTO}, $\overline{\mathcal{M}}_X$ has a finite number of connected components.
By definition of $\mathcal{M}_{\Lambda}$, there is a bijection between $O^+(\Lambda)/\Mon^2(X)$ and the set of connected components of $\overline{\mathcal{M}}_X$. It concludes the proof.
\end{proof}
We recall the main notation needed to state the Morrison--Kawamata cone conjecture (see for instance \cite[Section 6.4]{Markman11}).
\begin{defi}
Let $X$ be a projective irreducible symplectic orbifold. 
\begin{itemize}
\item[(i)]
A line bundle $L$ on $X$ is \emph{movable}, if the base locus of the linear system $|L|$ has codimension $\geq2$.
\item[(ii)]
The \emph{movable cone} $\Mov(X)$ is the cone in $\NS(X)\otimes \R$ spanned by the classes of movable line bundles.
\end{itemize}
\end{defi}
We set $T:=\Ima \rho$ and  $T_{bir}:=\Mon^2_{Bir}(X)/\Ker\rho$, with $\rho:\Mon^2_{Hdg}(X)\rightarrow O(\NS(X))$ the restriction morphism. 
Finally, we define $\overline{\Mov}^+(X)$ as the convex hull of $\overline{\Mov}(X)\cap(\NS(X)_{\Q})$ in $\NS(X)_{\R}$.

We also recall the definition of a fundamental domain.
\begin{defi}
Let $V$ be a vector space. Let $C\subset V$ be a cone. Let $\Gamma$ be a subgroup of $\GL(V)$ acting on $C$. Let $\Pi \subset C$ be a rational polyhedral cone. We say that $\Pi$ is a \emph{fundamental domain} for the action of $\Gamma$ on $C$ if :
\begin{itemize}
\item
$C=\bigcup_{g\in\Gamma}g(\Pi)$;
\item
for each $g\in\Gamma$, either $g(\Pi)=\Pi$ or $\Pi^{o}\cap g(\Pi)^{o}=\emptyset$, where $^o$ refers to the interior.
\end{itemize}
\end{defi}

\begin{prop}
Let $X$ be a projective irreducible symplectic orbifold. There exists a fundamental domain $\Pi$ for the action of $T_{bir}$ on $\overline{\Mov}^+(X)$.
\end{prop}
\begin{proof}
This result has been proved in \cite[Theorem 6.4]{Lehn2} when $b_2(X)\geq5$, so in this proof we assume that $b_2(X)=4$.
Moreover if $\rk\NS(X)=1$ there is nothing to prove, so we assume that $\rk\NS(X)=2$. In this case $\NS(X)\otimes \R=H^{1,1}(X,\R)$, so $\Amp(X)=\mathcal{K}_X$, where $\Amp(X)$ is the ample cone.
So we can prove that: 
\begin{equation}
\overline{\Mov}(X)=\overline{\mathcal{B}\mathcal{K}}_X. 
\label{MovBK}
\end{equation}
Indeed, by Lemma \ref{biracodim}, we have $\overline{\Mov}(X)\supset\overline{\mathcal{B}\mathcal{K}}_X$. We denote by $\Mov(X)^{o}$ the interior of the movable cone. Let $D\in \Mov(X)^{o}$; by \cite[Lemma 4.6]{Lehn2}, we have $(D,D')_q>0$ for any prime exceptional divisor $D'$.
Therefore, by Proposition \ref{bira}, we have $\overline{\Mov}(X)\subset\overline{\mathcal{B}\mathcal{K}}_X$.

Now, we can conclude using Lemmas \ref{simplechambers}, \ref{infiteray} and the same idea as explained in \cite[Proof of Theorem 6.25]{Markman11}. Let $x_0\in\Mov(X)^o$ and $\mathcal{C}^+_X$ be the convex hull of $\overline{\mathcal{C}}_X\cap(\NS(X)_{\Q})$ in $\NS(X)_{\R}$.
We set $$\Pi:=\left\{\left.x\in \mathcal{C}^+_X\right|\ (x_0,x)_q\leq(x_0,g(x))_q,\ \forall\ g\in T\right\}.$$
With the same argument as in \cite[Proof of Theorem 6.25]{Markman11} and Proposition \ref{finiteindex}, we have that $\Pi$ is a fundamental domain for the action of $T$ on $\mathcal{C}^+_X$.

If $\mathcal{C}_X$ contains finitely many birational Kähler-type chambers, according to Lemma \ref{simplechambers},
there are two cases $\mathcal{C}_X$ contains two or one birational Kähler-type chambers.
If $\mathcal{C}_X$ contains only one birational Kähler-type chambers, then $\mathcal{B}\mathcal{K}_X=\mathcal{C}_X$ and $T=T_{bir}$.
So we conclude with (\ref{MovBK}).
If $\mathcal{C}_X$ contains two birational Kähler-type chambers, we know by Lemma \ref{simplechambers} (ii) that $T=\left\langle R_B\right\rangle$ and $T_{bir}=\left\langle \id\right\rangle$, with $B\in\cW_X^+$. So using (\ref{MovBK}), it remains to show that $\Pi\subset\overline{\mathcal{B}\mathcal{K}}_X$. We recall from the proof of Lemma \ref{simplechambers} that an extremal ray of $\overline{\mathcal{B}\mathcal{K}_X}$ is contained in $B^{\bot}$ and the other extremal ray of $\overline{\mathcal{B}\mathcal{K}_X}$ coincides with an extremal ray of $\mathcal{C}_X$. Let $x\in\Pi$; therefore, we have to verify that $(x,B)_q\geq0$. We have $(x_0,x)_q\leq (x_0,R_B(x))_q$. That is $(x_0,x)_q\leq (x_0,x)_q-2\frac{(B,x)_q}{q(B)}(x_0,B)_q$. Since $(x_0,B)_q>0$ and $q(B)<0$, we have $(B,x)_q\geq0$.

If $\mathcal{C}_X$ contains infinitely many birational Kähler-type chambers, then by Lemma \ref{infiteray} the two extremal rays of $\overline{\BK}_X$ are precisely the rays lying in $D_1^\perp$ and $D_2^\perp$, where $D_1$ and $D_2$ are prime exceptional divisors.  Repeating the previous argument, the inequality
\[
(x_0,x)_q \;\le\;(x_0,R_{D_i}(x))_q
\]
implies $(D_i,x)_q\ge0$, and hence $\Pi\subset\overline{\BK}_X$.  Moreover, if $g\in T$ satisfies $g(\Pi)\cap\BK_X\neq\emptyset$, then Theorem \ref{main}(ii) forces $g\in T_{bir}$, so in fact $g(\Pi)\subset\overline{\BK}_X$.  Since $\Pi$ is already a fundamental domain for the action of $T$ on $\mathcal{C}_X^+$, it follows that $\Pi$ is also a fundamental domain for the action of $T_{bir}$ on $\overline{\Mov}(X)$.
\end{proof}


\section{Application to mirror symmetry}\label{mir}
\subsection{Motivation and some previous results}
In this section, we propose a definition of the mirror symmetry for an irreducible symplectic orbifold endowed with a Kähler class. The main idea of mirror symmetry is to exchange the holomorphic
2-form and the Kähler metric. Several definitions of mirror symmetry are possible.
We follow the one chosen by Huybrechts in the framework of hyperkähler manifolds \cite[Section
  6.4]{Huybrechts2}, which in turn is inspired by Aspinwall and Morrison \cite{Aspinwall}. The definition of
Huybrechts is given by an involution on a period subdomain. When Huybrechts proposed his definition, the
knowledge on the Kähler cone for irreducible symplectic manifolds was not yet developed sufficiently to deduce
the mirror symmetry for these manifolds in general. 
Building on the results on the Kähler cone (Section \ref{sec:wall-divs}), we are able to complete this missing step in the orbifold setting (a fortiori in the smooth setting).
More precisely, via the period map, we provide an isomorphism between the period domain
$\widetilde{\mathcal{D}_\Lambda}$ and a moduli space  $\widetilde{\mathcal{M}_{\Lambda}}$ of irreducible
symplectic orbifolds endowed with a Kähler class and some additional data (Theorem \ref{TorelliMirror}); it allows us to extend Huybrechts' definition to a subspace of $\widetilde{\mathcal{M}_{\Lambda}}$ in Section \ref{dms}.


Another approach to mirror symmetry for K3 surfaces was studied by \cite{Dolgachev}, who provides a mirror symmetry between moduli spaces of lattice-polarized
K3 surfaces (see Definition \ref{polarized}).  Dolgachev's definition has been inspired by mathematical
\cite{Aspinwall}, \cite{Giveon}, \cite{Martinec}, \cite{Roan}, \cite{Voisin} and physical literature
\cite{Borcea}, \cite{Kobayashi}, \cite{Martinec}.  The results of Dolgachev for K3 surfaces have been generalized by Camere in \cite{Camere-2016} to hyperkähler
manifolds.
We will prove in Proposition \ref{polarized} that Huybrechts definition coincides with
the one of Dolgachev and Camere.

In addition, our work is a generalization of \cite[Section 4.1 and 4.2]{Franco-Jardim-Menet}, where a weaker
version of Theorem \ref{TorelliMirror} was obtained in the smooth case.

Section \ref{GTTOK} is devoted to the proof of Theorem \ref{TorelliMirror}, which uses our previous results on the global Torelli theorem (Theorem \ref{main}) and the Kähler cone (Section \ref{sec:wall-divs}). In Section \ref{dms}, we discuss the definition of mirror symmetry obtained by combining Theorem \ref{TorelliMirror} and Huybrechts' definition \cite[Section 6.4]{Huybrechts2}.

For simplicity of the equations in this section, we denote the \emph{Beauville--Bogomolov form} simply by the dot "$\cdot$".

\subsection{Global Torelli theorem for marked irreducible symplectic orbifolds endowed with a Kähler
  class}\label{GTTOK}

Let $\Lambda$ be a lattice of signature $(3,\rk \Lambda-3)$.  For a marked irreducible symplectic
orbifold  $(X,\varphi)$, we denote by $\left[X,\varphi\right]$ its class of isomorphism.  Fix a connected
component $\mathcal{M}_\Lambda^{\circ}$ of the moduli space of marked irreducible symplectic orbifolds of
Beauville--Bogomolov lattice $\Lambda$.

In this section, we consider quadruplets $(X,\varphi,\sigma_X,\omega_X,\beta),$
where $(X,\varphi)$ is a marked irreducible symplectic orbifold, $\sigma_X\in H^{2,0}(X)$, $\omega_X\in\mathcal{K}_X$ and $\beta\in H^2(X,\R)$. The class $\beta$ is called a \emph{$B$-field}. 
We define
$$\widetilde{\mathcal{M}_\Lambda}:=\left\{\left.(X,\varphi,\sigma_X,\omega_X,\beta)\ \right|\ \left[X,\varphi\right]\in\mathcal{M}_\Lambda^{\circ},\ 0
\neq \sigma_X\in H^{2,0}(X),\  \omega_X\in\mathcal{K}_X,\ \beta\in H^2(X,\R)\right\}_{\diagup\sim},$$
where $\sim$ is the identification of isomorphic objects, i.e.~$(X,\varphi,\sigma_X,\omega_X,\beta)\sim (X',\varphi',\sigma_{X'},\omega_{X'},\beta')$
if and only if there exists an isomorphism $f:X\rightarrow X'$ such that $f^*=\varphi^{-1}\circ\varphi'$,
$f^*(\sigma_{X'})=\sigma_{X}$, $f^*(\omega_{X'})=\omega_{X}$, and $f^*(\beta')=\beta$. We denote by
$\cW_\Lambda\subset \Lambda$ the set from Definition \ref{def:setW}, which gives the wall divisors of orbifolds in $\mathcal{M}_\Lambda^{\circ}$ via the marking.

\begin{rmk}
The set $\widetilde{\mathcal{M}_\Lambda}$ is endowed with a structure of differential manifold, obtained as an open submanifold of 
$$\doublewidetilde{\mathcal{M}_\Lambda}:=\left\{\left.(X,\varphi,\sigma_X,\omega_X,\beta)\ \right|\ \left[X,\varphi\right]\in\mathcal{M}_\Lambda^{\circ},\ 0\neq \sigma_X\in
H^{2,0}(X),\  \omega_X\in H^{1,1}(X,\R),\ \beta\in H^2(X,\R)\right\}_{\diagup\sim},$$
where $\sim$ is identification of isomorphic objects, as before.
The structure of differential manifold of $\doublewidetilde{\mathcal{M}_\Lambda}$ is given by the
period map, which is locally a bijection, by the local Torelli Theorem \cite[Theorem 3.17]{Menet-2020}: 
$$\xymatrix@R0cm@C0.5cm{\ \ \ \ \ \ \ \ {\doublewidetilde{\mathscr{P}}}:&
  {\doublewidetilde{\mathcal{M}_\Lambda}}\ar[r]& \dG_\Lambda
  \subseteq (\Lambda \otimes \bC)^2\\
&(X,\varphi,\sigma_X,\omega_X,\beta)\ar[r]&\left(\varphi(\sigma_X),\varphi(\beta+i\omega_X)\right),}$$
with 
\begin{equation}
\dG_\Lambda:=\left\{\left.(\alpha,\beta+ix)\in
(\Lambda\otimes\C)^2\ \right|\ \alpha^2=0,\ \alpha\cdot\overline{\alpha}>0,\ \alpha\cdot x=0,\  x ^2>0
\right\}.
\label{last}
\end{equation}
In \cite[Section 4.4]{Huybrechts2}, the space considered is slightly different: 
$$\Gr_{2,1}^{po}(\Lambda_{\R})\times \Lambda_{\R}:=\left\{\left.(\alpha,\beta+ix)\in
\mathbb{P}(\Lambda\otimes\C)\times\Lambda\otimes\C \ \right|\ \alpha^2=0,\ \alpha\cdot\overline{\alpha}>0,\ \alpha\cdot x=0,\  x ^2>0
\right\};$$
it can be obtained as a quotient of $\dG_\Lambda$.
\end{rmk}

\begin{rmk}
If we choose to restrict to the case where $\beta\in H^{1,1}(X,\R)$, then 
$$\doublewidetilde{\mathcal{M}_\Lambda}\,^{1,1}:=\left\{\left.(X,\varphi,\sigma_X,\omega_X,\beta)\ \right|\ \left[X,\varphi\right]\in\mathcal{M}_\Lambda^{\circ},\ 0\neq\sigma_X\in H^{2,0}(X),\  \omega_X\in H^{1,1}(X,\R),\ \beta\in H^{1,1}(X,\R)\right\}_{\diagup\sim}$$
can be endowed with the structure of a complex manifold, inherited from the period map (compare \cite[Chapter 1 Section 2]{Magnusson}).
\end{rmk}
We can generalize the global Torelli Theorem \cite[Theorem 1.1]{Menet-2020} to $\widetilde{\mathcal{M}_\Lambda}$.
\begin{defi} Define the following generalized period domain: $$\widetilde{\mathcal{D}_\Lambda}:=\left\{\left.(\alpha,\beta+ix)\in
  (\Lambda\otimes\C)^2\ \right|\ \alpha^2=0,\ \alpha\cdot\overline{\alpha}>0,\ x^2>0,\ \alpha\cdot
  x=0,\ \alpha^\bot\cap x^\bot\cap \cW_\Lambda=\emptyset\right\} \subseteq \dG_\Lambda.$$ 
\end{defi}
\begin{thm}\label{TorelliMirror} Assume that $\mathcal{M}_\Lambda^{\circ}$ is non-empty.
The period domain $\widetilde{\mathcal{D}_\Lambda}$ has two connected components
$\widetilde{\mathcal{D}}_1$ and $\widetilde{\mathcal{D}}_2$, and there exists $i\in \left\{1,2\right\}$ such that the period map
$$\xymatrix@R0cm@C0.5cm{\ \ \ \ \ \ \ \ \widetilde{\mathscr{P}}:& \widetilde{\mathcal{M}_{\Lambda}}\ar[r]& \widetilde{\mathcal{D}}_i\\
&(X,\varphi,\sigma_X,\omega_X,\beta)\ar[r]&\left(\varphi(\sigma_X),\varphi(\beta+i\omega_X)\right),}$$
is an isomorphism.
\end{thm}
\begin{proof}
The end of this section is devoted to the proof of this theorem.
\subsubsection*{Step 1: The set $\widetilde{\mathcal{D}_\Lambda}$ has at least two connected components}
Consider $(\alpha, \beta + ix) \in \widetilde{\dD_\Lambda}$. By the surjectivity of the usual period map $\cP$
(see  \cite[Proposition 5.8]{Menet-2020}), there exists a marked irreducible symplectic orbifold $(X_\alpha,\phi_\alpha) \in \dM^\circ_\Lambda$ such that
$\alpha \in \phi_\alpha(H^{2,0}(X_\alpha))$. Since by assumption $x^2>0$, this implies that either $\phi_\alpha^{-1}(x)\in \dC_{X_\alpha}$
(i.e. $x\cdot \phi_\alpha(K_{X_\alpha}) >0$), or $-\phi_\alpha^{-1}(x) \in \dC_{X_\alpha}$.
By continuity, this splits $\widetilde{\dD_\Lambda}$ into two disjoint open sets:
$$\widetilde{\dD_1}\coloneqq \{(\alpha, \beta+ix) \,|\, \phi_\alpha^{-1}(x)\in \dC_{X_\alpha}\}, \quad \text{and} \quad
\widetilde{\dD_2}\coloneqq \{(\alpha, \beta+ix) \,|\, -\phi_\alpha^{-1}(x)\in  \dC_{X_\alpha}\}.$$
Note that, by exchanging $x$ and $-x$, one verifies that both sets are non-empty.

\subsubsection*{Step 2: The moduli space $\widetilde{\mathcal{M}_\Lambda}$ is connected}
Let $(X,\varphi,\sigma_X,\omega_X,\beta)$ and $(Y,\psi,\sigma_{Y},\omega_{Y},\gamma)$ be two elements in $\widetilde{\mathcal{M}_\Lambda}$.
By Lemma \ref{connected}, we can connect $(X,\phi)$ and $(Y,\psi)$ by twistor spaces. The first twistor
line is given by the positive 3-space $W\coloneqq \Vect(\phi(\Rea\sigma_X), \phi(\Ima\sigma_X), \omega_X')$ for some
$\omega_X'\in \mathcal{K}_X$. Since any fiber $X_1$ of the twistor family is endowed with an induced marking $\phi_1$
and a canonical Kähler class
$\omega_1$, one can connect $(X, \phi, \sigma_X, \omega_X', \beta)$ to $(X_1, \phi_1, \sigma_{1}, \omega_1,
\phi^{-1}\circ \phi (\beta))$ for some $\sigma_1 \in H^{2,0}(X_1).$ Use that $\mathcal{K}_X$ is connected to observe
that this is actually connected to the original $(X,\phi, \sigma_X, \omega_X, \beta)$. 
Repeating this process for the other twistor spaces, we can find $\sigma_{Y}'$ and $\omega_{Y}'$ such that 
$(X,\varphi,\sigma_X,\omega_X,\beta)$ and $(Y,\psi,\sigma_{Y}',\omega_{Y}',\psi^{-1}\circ\varphi(\beta))$ are connected.
Using again that the Kähler cone, the space of non-zero holomorphic 2-forms, and the second cohomology group with real coefficient are connected, it follows that
$(X,\varphi,\sigma_X,\omega_X,\beta)$ and $(Y,\psi,\sigma_{Y},\omega_{Y},\gamma)$ can be connected in $\widetilde{\dM_\Lambda}$.



We prove now that $\widetilde{\mathscr{P}}:\widetilde{\mathcal{M}_{\Lambda}}\rightarrow\widetilde{\mathcal{D}}_i$ is an isomorphism.
\subsubsection*{Step 3: The map $\widetilde{\mathscr{P}}$ is injective}
Indeed, choose $(X,\varphi,\sigma_X,\omega_X,\beta)$ and $(X',\varphi',\sigma_{X'},\omega_{X'},\beta')$ in
$\widetilde{\mathcal{M}_{\Lambda}}$. Since $(X,\phi)$ and $(X',\phi')\in \dM_\Lambda^\circ$ are deformation equivalent,  $$\varphi'^{-1}\circ\varphi:H^2(X,\Z)\rightarrow H^2(X',\Z)$$
defines a parallel transport operator.
Assume that: 
$$\widetilde{\mathscr{P}}(X,\varphi,\sigma_X,\omega_X,\beta)=\widetilde{\mathscr{P}}(X',\varphi',\sigma_{X'},\omega_{X'},\beta').$$
Then $\phi'^{-1}\circ \phi$ is a Hodge isometry which sends a Kähler class to a Kähler class. Hence, by
Theorem \ref{main} (ii), there exists an isomorphism $f:X'\rightarrow X$ such that
$f^*=\varphi'^{-1}\circ\varphi$, which means that
$(X,\varphi,\sigma_X,\omega_X,\beta)\simeq(X',\varphi',\sigma_{X'},\omega_{X'},\beta')$ are isomorphic.
\subsubsection*{Step 4: The map $\widetilde{\mathscr{P}}$ is surjective}
By the convention from Step 1, $\widetilde{\cP}(\widetilde{\dM_\Lambda})\subseteq \widetilde{\dD_1}$.
Let $(\alpha,\beta+ix)\in \widetilde{\mathcal{D}_1}$.  Theorem \ref{mainGTTO} implies that there exists $(X,\varphi)\in
\mathcal{M}_\Lambda^{\circ}$ such that $\varphi^{-1}(\alpha)\in H^{2,0}(X)$. Then $\varphi^{-1}(\beta)\in
H^2(X,\R)$. It remains to study $\varphi^{-1}(x)$, which is an element of $\dC_X$ by definition of
$\widetilde{\dD_1}$.
By Corollary \ref{wall}, the wall divisors on $X$ are given by
$\phi^{-1}(\cW_\Lambda)\cap \Pic(X)$, and therefore $\phi^{-1}(x)\cdot D\neq 0$, for all wall divisors $D\in
\cW_X$, by the definition of $\widetilde{\mathcal{D}_\Lambda}$.

Hence, by Corollary \ref{cor:alphamovedtoK} there exists $g\in \Mon^2_{Hdg}(X)$ and a bimeromorphic map $f:X\dashrightarrow Y$such that $ \varphi^{-1}(x)\in g(f^*(\mathcal{K}_Y))$. 
We consider $\psi:=\varphi\circ g\circ f^*$ (note that by Proposition \ref{mainprop2} $f^*$ is a parallel transport operator). It follows that $\psi^{-1}(x)\in \mathcal{K}_Y$ and $\psi^{-1}(\alpha)\in H^{2,0}(Y)$.

Then $(Y,\psi,\psi^{-1}(\alpha),\psi^{-1}(x),\psi^{-1}(\beta))\in\widetilde{\mathcal{M}_{\Lambda}}$ and $\mathcal{P}(Y,\psi,\psi^{-1}(\alpha),\psi^{-1}(x),\psi^{-1}(\beta))=(\alpha,\beta+ix)$.
\end{proof}
\begin{rmk}
Of course, the previous theorem remains true if we remove the data of the B-field $\beta$ in $\widetilde{\mathcal{M}_{\Lambda}}$ and in $\widetilde{\mathcal{D}}_i$. This data will be useful only in the framework of mirror symmetry. 
\end{rmk}
\begin{rmk}
Theorem \ref{TorelliMirror} also shows that $\widetilde{\mathcal{M}_\Lambda}$ is separated.
\end{rmk}
\begin{rmk} \label{rmk:conncompsM}
There is a natural isomorphism between $\widetilde{\mathcal{D}}_1$ and $\widetilde{\mathcal{D}}_2$, given by $-\id$:
$$\xymatrix@R0cm@C0.5cm{\ \ \ \ \ \ \ \ \widetilde{\mathcal{D}}_1\ar[r]& \widetilde{\mathcal{D}}_2\\
\left(\alpha,\beta+ix\right)\ar[r]&\left(-\alpha,-\beta-ix\right).}$$
Moreover, if $\widetilde{\mathscr{P}}:\widetilde{\mathcal{M}_\Lambda}\rightarrow \widetilde{\mathcal{D}}_1$ is an isomorphism, we can consider
$$\widetilde{\mathcal{M}_\Lambda^-}:=\left\{\left.(X,\varphi,\sigma_X,\omega_X,\beta)\ \right|\ (X,-\varphi,\sigma_X,\omega_X,\beta)\in \widetilde{\mathcal{M}_\Lambda}\right\}$$
and the induced map ${\widetilde{\mathscr{P}}}: {\widetilde{\mathcal{M}_\Lambda^-}}\rightarrow \widetilde{\mathcal{D}}_2$ is an isomorphism.
\end{rmk}
\begin{rmk}\label{rem:conncompviadet}
  An alternative way to determine the connected component $\widetilde{\dD}_i$ of an element
  $(\alpha, \beta + ix)\in \widetilde{\dD_\Lambda}$ is the following. Fix a positive definite 3-space $W\subseteq
  \Lambda_\bR$ with a determinant form $\det_W$. Let $\pi_W\colon \Lambda_\bR \to W$ be the orthogonal projection to $W$.
  Note that $\Vect_\bR(\Rea \alpha, \Ima \alpha, x)\subseteq \Lambda_\bR$ is also a positive definite
  3-space.

  Since the signature of $\Lambda$ is $(3,\rk\Lambda-3)$, this implies that $\pi_W\colon
  \Vect_\bR(\Rea \alpha, \Ima \alpha, x) \to W$ is an isomorphism, and therefore
  $\det_W (\pi_W(\Rea \alpha), \pi_W(\Ima \alpha),\pi_W (x))\neq 0 $.
  Therefore, by continuity, the sign of $\det_W (\pi_W(\Rea \alpha), \pi_W(\Ima \alpha),\pi_W (x))$
  determines the connected component of $(\alpha, \beta + ix)$, since both signs are achieved by elements of
  $\widetilde{\dD_\Lambda}$ (consider
  e.g. $(\alpha, -\beta- ix)$).
\end{rmk}
\subsection{Definition of the mirror symmetry}\label{dms}
Using Theorem \ref{TorelliMirror}, we can define algebraically a mirror symmetry on a subset of
the period domain $\widetilde{\mathcal{D}_\Lambda}$, which will induce a mirror symmetry on a corresponding subset of the moduli
space $\widetilde{\mathcal{M}_{\Lambda}}$.


Let $\Lambda$ be a lattice of signature $(3,\rk\Lambda-3)$. For  any field $K$, we
denote: $$\Lambda_{K}:=\Lambda\otimes K.$$
For $n\in \mathbb{N}^*$, let $U(n)$ be the lattice of rank 2 with basis $(v,v^*)$ such that $v^2=v^{*2}=0$ and
$v\cdot v^*=n$.
Assume that there exists a primitive embedding $j:U(n)\hookrightarrow \Lambda$, such that the sublattice $j(U(n))$ is a
direct summand, i.e.~$\Lambda= \Lambda'\oplus^{\bot} j(U(n))$. For simplicity of the notation, we also
denote by $(v,v^*)$ the basis of $j(U(n))$ and when there is no ambiguity, we simply write $U(n)$ for
$j(U(n))$.

Let $U'(n)$ be another hyperbolic lattice isometric to $U(n)$. We denote by $\xi_j$ the isometry of $O(\Lambda\oplus U'(n))$ which fixes $\Lambda'$ and exchanges $U(n)$ and $U'(n)$. As explained in \cite[Section 6]{Huybrechts2}, the mirror map is well defined on the following period domain:
\begin{align*}
  \Gr_{2,2}^{po}&(\Lambda_{\R}\oplus U'(n)_{\R})\\&:=
  \left\{\left.(\alpha_1,\alpha_2)\in \big(\bP(\Lambda_\bC\oplus
U'(n)_\bC)\big)^2\ \right|\ \alpha_1^2=\alpha_2^2=\alpha_1\cdot\alpha_2=0,\ \alpha_1\cdot\overline{\alpha_1}>0,\ \alpha_2\cdot\overline{\alpha_2}>0\right\}.
\end{align*}
The notation $\Gr_{2,2}^{po}(\Lambda_{\R}\oplus U'(n)_{\R})$ was chosen, since this space can be identified
with the Grassmannian which parametrizes pairs of orthogonal positive 2-planes in $\Lambda_{\R}\oplus U'(n)_{\R}$.
On $\Gr_{2,2}^{po}(\Lambda_{\R}\oplus U'(n)_{\R})$ the mirror map is given by: $$\overline{m_j}:=\iota\circ \xi_j,$$ where 
 $\iota(\alpha_1,\alpha_2)=(\alpha_2,\alpha_1)$.
 \begin{rmk}
 The general concept of mirror symmetry is to exchange the holomorphic 2-form and the metric. This operation will
 roughly be given by $\iota$. We refer to \cite{Aspinwall} for the physical meaning of the composition by $\xi_j$.
 \end{rmk}
\begin{rmk}
In \cite[Section 6]{Huybrechts2} it is considered the case $n=1$. However, this condition is too restrictive in
the framework of orbifolds (see Example \ref{exempleBB}). This is why we allow any $n\in\mathbb{N}^*$. This does not affect the definition, since $U(n)_{\R}=U_{\R}$.
\end{rmk}
In the following, we consider a lift of $\overline{m_j}$ to a subset of $\widetilde{\mathcal{D}_\Lambda}$. We recall that $\dG_\Lambda$ is defined in (\ref{last}).
It is shown in \cite[Section 4]{Huybrechts2} that $\dG_\Lambda$ (which contains
$\widetilde{\mathcal{D}_\Lambda}$) admits a natural map to $\Gr_{2,2}^{po}(\Lambda_{\R}\oplus U'(n)_{\R})$, obtained as the composition of the quotient map $Q:\dG_\Lambda\rightarrow \Gr_{2,1}^{po}(\Lambda_{\R})\times \Lambda_{\R}$ and the following natural embedding:

$$\xymatrix@R0cm@C0.5cm{\overline{h}:&\Gr_{2,1}^{po}(\Lambda_{\R})\times \Lambda_{\R}\ar@{^{(}->}[r]& \Gr_{2,2}^{po}(\Lambda_{\R}\oplus U'(n)_{\R})\\
&(\alpha,\beta+ix)\ar[r]&\left(\sqrt{n}\alpha-(\alpha\cdot\beta)w,\sqrt{n}\beta+\frac{1}{2}(x^2-\beta^2)w+w^*+i\left(\sqrt{n}x-(x\cdot\beta)w\right)\right),}$$
where $(w,w^*)$ is a basis of $U'(n)$; we set $h:=\overline{h}\circ Q$.
Unfortunately, $h(\dG_\Lambda)$ is not fixed by $\overline{m_j}$. We will therefore consider a subspace of $\dG_\Lambda$
such that its image under $h$ is fixed by $\overline{m_j}$: 
\small
$$\Dom(m_j)^\dG:=\left\{\left.(\alpha,\beta+ix)\in \dG_\Lambda\ \right|\ \Ima(\alpha)\cdot v=0,\ \Ree(\alpha)\cdot v\neq0,\ x\cdot v=\ \beta\cdot v=0\right\}.$$
\normalsize
Let $\pr: \Lambda\rightarrow j(U(n))^\bot$ be the projection. Then, similar to \cite[Proposition
  6.8]{Huybrechts2}, we can define the following action  $m_j$ on $\Dom(m_j)^\dG$:
\begin{equation}
\xymatrix@R0cm@C0.5cm{\ \ \ \ \ \ \ \ m_{j}:& \Dom(m_j)^\dG\ar[r]&  \Dom(m_j)^\dG\ \ \ \ \ \\
&(\alpha,\beta+ix)\ar@{|->}[r]&\left(\frac{\sqrt{n}\pr(\beta+ix)-\frac{1}{2}(\beta+ix)^2v+v^*}{\Ree(\alpha)\cdot v}
  ,\frac{\sqrt{n}\pr(\alpha)-(\alpha\cdot\beta)v}{\Ree(\alpha)\cdot v}\right).}
\label{equationmiror}
\end{equation}
\begin{prop}
  The map $m_j$ is an involution, which satisfies 
  $$h\circ m_j = \overline{m_j} \circ h.$$
\end{prop}
\begin{proof}
  The claimed compatibility with $\overline{m_j}$ is an immediate consequence of  \cite[Proposition
    6.8]{Huybrechts2}. Note that there are
  only two differences between the map $m_j$ and the map which is considered in \cite{Huybrechts2}. The first
  difference is that in {\it loc.~cit.}
  the domain of $m_j$ is a quotient of $\dG_\Lambda$. The second is that we performed a change of
  variables, by replacing $v$ by $\frac{1}{\sqrt{n}}v$ (and similarly for $v^*$, $w$, and $w^*$),
  in order to obtain the same intersection values. We can verify that $m_j$ is an involution by a direct computation.
	We denote $(\alpha^{\vee\vee},\beta^{\vee\vee}+ix^{\vee\vee}):=m_j\circ m_j(\alpha,\beta+ix)$. We are going to check that 
	$\beta^{\vee\vee}+ix^{\vee\vee}=\beta+ix$; the verification that $\alpha^{\vee\vee}=\alpha$ is very similar and is left to the reader.
  By definition of $m_j$, we have:
	\begin{align*}
	\beta^{\vee\vee}+ix^{\vee\vee}&=\frac{\frac{n\pr(\beta+ix)}{\Ree(\alpha)\cdot
            v}-\frac{\left[\left(\sqrt{n}\pr(\beta+ix)-\frac{1}{2}(\beta+ix)^2v+v^*\right)\cdot\left(\sqrt{n}\pr(\Ree(\alpha))-(\Ree(\alpha)\cdot\beta)v\right)\right]}{(\Ree(\alpha)\cdot
            v)^2}v}
             {n(\Ree (\alpha)\cdot v)^{-1}}\\
	&=\pr(\beta+ix)-\frac{\pr(\beta+ix)\cdot\pr(\Ree \alpha)-\Ree(\alpha)\cdot\beta}{\Ree(\alpha)\cdot v}v.
	\end{align*}
	It remains to compute $\pr(\beta+ix)\cdot\pr(\Ree \alpha)$:
	\begin{align*}
	\pr(\beta+ix)\cdot\pr(\Ree \alpha)&=\left(\beta+ix-\frac{(\beta+ix)\cdot v^*}{n}v\right)\cdot \left(\Ree(\alpha)-\frac{\Ree(\alpha)\cdot v^*}{n}v-\frac{\Ree(\alpha)\cdot v}{n}v^*\right)\\
	&=\beta\cdot \Ree(\alpha)-\frac{(\beta+ix)\cdot v^*}{n}\left(v\cdot \Ree(\alpha)\right).
	\end{align*}
	Combined with the previous equation, we obtain:
	$$\beta^{\vee\vee}+ix^{\vee\vee}=\pr(\beta+ix)+\frac{(\beta+ix)\cdot v^*}{n}v=\beta+ix.$$
\end{proof}
Finally, $m_j$ will be well defined on the following period subdomain of $\widetilde{\mathcal{D}_\Lambda}$:
For $(\alpha,\beta+ ix)\in \Dom(m_j)^\dG$, choose the notation $(\alpha^\vee, \beta^\vee, + i x
^\vee)\coloneqq m_j(\alpha, \beta+ ix)$. Then set
$$\Dom(m_j):=\left\{\left.(\alpha,\beta+ix)\in\widetilde{\mathcal{D}_\Lambda}\  \right|\ \Ima(\alpha)\cdot v=0,\ \Ree(\alpha)\cdot v\neq0,\ x\cdot v=\ \beta\cdot v=0,\ \alpha^{\vee\bot}\cap x^{\vee\bot}\cap \cW_\Lambda=\emptyset\right\}.$$

Moreover, we denote:
$$\Dom(m_j)_1:=\Dom(m_j)\cap\widetilde{\mathcal{D}}_1\quad \text{and}\quad \Dom(m_j)_2:=\Dom(m_j)\cap\widetilde{\mathcal{D}}_2.$$
We obtain the following proposition.
\begin{prop}\label{proporientation}
The mirror involution $m_j$ exchanges $\Dom(m_j)_1$ and $\Dom(m_j)_2$.
\end{prop}
\begin{proof}
  Using Remark \ref{rem:conncompviadet}, we can fix a convenient positive definite three-space $W
  \subseteq \Lambda_\bR$ with a volume form $\det_W$, and we only need to compare the signs of
  $\det_W (\pi_W(\Rea \alpha), \pi_W(\Ima \alpha),\pi_W (x))$
  and $\det_W (\pi_W(\Rea \alpha^\vee), \pi_W(\Ima \alpha^\vee),\pi_W (x^\vee))$.

  Let us choose  $W=\Vect_{\R}(\frac{1}{2}(v+v^*),\pr(\Ima \alpha), \pr(x))$ with
  $\det_W(\frac{1}{2}(v+v^*),\pr(\Ima \alpha), \pr(x))=1$.
  
  We start by determining $\det_W(\pi_W(\Rea \alpha),\pi_W(\Ima \alpha), \pi_W(x))$.
  Notice that
  $$\pi_W(\Ima \alpha)=
  \pi_W(\pr \Ima \alpha)
  + \frac{1}{n}(\Ima \alpha\cdot v^*)\pi_W(v)
  =\pr \Ima \alpha + \frac{1}{n}(\Ima\alpha \cdot v^*)\frac {1}{2}(v+v^*)$$
   and similarly
   $\pi_W(\Rea \alpha)= 0  + \frac{1}{n}[\Rea \alpha \cdot (v  + v^*)] \frac{1}{2}(v+v^*),$ and
   $\pi_W(x)=\pr x + \frac{1}{n}(x\cdot v^*)\frac{1}{2}(v+v^*)$.
   Therefore,
   \begin{align*}
     \det\hspace{-2pt}\hspace{1pt}_W(\pi_W(\Rea \alpha),\pi_W(\Ima \alpha), \pi_W(x))&=
     \det
     \begin{pmatrix}\frac{1}{n}(\Rea \alpha \cdot (v+v^*)) & \frac{1}{n}(\Ima \alpha \cdot v^*) &
       \frac{1}{n}(x \cdot v^*) \\
       0 & 1 & 0\\
       0 & 0 & 1
     \end{pmatrix}
     \\
     &= \frac{1}{n}\Rea \alpha \cdot (v+v^*).
   \end{align*}
Since $\Rea \alpha^2>0$, $(\Rea \alpha)\cdot v$ and $(\Rea \alpha)\cdot v^*$ have the same sign.
In particular the sign of $\frac{1}{n}\Rea \alpha\cdot (v+ v^*)$ is the same as the sign of $\Rea \alpha \cdot v$.

As a second step, we need to determine the sign of
$\det_W(\pi_W(\Rea \alpha^\vee),\pi_W(\Ima \alpha^\vee), \pi_W(x^\vee))$.
Note, that by replacing $\beta$ with $0$, one obtains an element in the same connected component of
$\Dom(m_j)$, so it is sufficient to treat the special case $\beta=0$.
Using \eqref{equationmiror} and the same reasoning as above, we need to determine the determinant of the
following matrix:
\begin{equation*}
  M\coloneqq 
  \frac{1}{\Rea \alpha \cdot v}
  \begin{pmatrix}
    \frac{1}{2}x^2 + 1 &0 &0 \\
    0 & 0 & \sqrt{n}\\
    0 & \sqrt{n} & 0 
  \end{pmatrix}.
\end{equation*}
Observe that $\det (M) = -\frac{n}{(\Rea \alpha \cdot v)^3}(\frac{1}{2} x^2 +1)$. Since $x^2>0$, this has the
same sign as $-\Rea \alpha \cdot v$.
Therefore $(\alpha, \beta + ix)$ and $m_j(\alpha,\beta + ix)=(\alpha^\vee,\beta^\vee + ix^\vee)$ lie in
different connected components of $\widetilde{\dD_\Lambda}$, which proves that $m_j$ interchanges
$\Dom(m_j)_1$ and $\Dom(m_j)_2$.
\end{proof}
\begin{remark}
  Proposition \ref{proporientation} can colloquially be rephrased by saying that the mirror map changes the orientation of positive
  three-spaces, which happens to match our experience with real life mirrors.
\end{remark}

To be able to have a better idea in practice of the elements contained in $\Dom(m_j)$, we provide the following lemma, which is a straightforward consequence of \cite[Proposition 4.6]{Franco-Jardim-Menet}.
\begin{lemme}\label{domain}
Let $(\alpha,\beta+ix)\in\widetilde{\mathcal{D}_\Lambda}$ such that:
\begin{itemize}
\item 
$(\alpha,\beta+ix)\in \Dom(m_j)^{\mathcal{G}}$,
\item
$x^2\in \R\smallsetminus \Q$,
\item
  For each $D\in \cW_\Lambda$ one of the following two implications is satisfied:\newline
  Either
$\pr(\Ree \alpha)\cdot D\neq 0\ \Rightarrow\ \pr(\beta)\cdot D\neq 0$ or $\alpha \cdot D\neq 0\ \Rightarrow \Ima\alpha \cdot D\neq 0$.
\end{itemize}
Then there exists a dense uncountable subset $\Pi\in \R^*$, such that for all $\lambda\in \Pi$, we have $(\alpha,\lambda\beta+ix)\in \Dom(m_j)$.
\end{lemme}
We are now ready to state the definition of the mirror symmetry of a marked irreducible symplectic orbifold
endowed with a Kähler class and a B-field.
\begin{defi}
Let $\widetilde{\mathscr{P}}:\widetilde{\mathcal{M}_{\Lambda}}\rightarrow\widetilde{\mathcal{D}}_1$ be the
isomorphism provided by Theorem \ref{TorelliMirror}, and consider the associated natural map 
${\widetilde{\mathscr{P}}}: {\widetilde{\mathcal{M}_\Lambda^-}}\rightarrow \widetilde{\mathcal{D}}_2$ 
from Remark \ref{rmk:conncompsM}.
Consider the moduli subspaces
$$\widetilde{\mathcal{M}_{\Lambda}}^j:=\widetilde{\mathscr{P}}^{-1}(\Dom(m_j)_1) \text{\ and \ }
\widetilde{\mathcal{M}_{\Lambda}^-}^j:=\widetilde{\mathscr{P}}^{-1}(\Dom(m_j)_2).$$
The mirror symmetry $\widetilde{m}_j$ is an involution defined on the moduli space
$\widetilde{\mathcal{M}_{\Lambda}}^j \cup \widetilde{\mathcal{M}_{\Lambda}^-}^j$
by:
$$\widetilde{m}_j:=\widetilde{\mathscr{P}}^{-1}\circ m_j\circ\widetilde{\mathscr{P}}.$$
Since $m_j$ interchanges $\widetilde{\dD}_1$ and  $\widetilde{\dD}_2$, the map $\widetilde{m}_j$ exchanges $\widetilde{\mathcal{M}_{\Lambda}}^j$ and $\widetilde{\mathcal{M}_{\Lambda}^-}^j$.
\end{defi}

Actually, a mirror can be defined for all the elements of $\left\{\left.(X,\varphi)\in \mathcal{M}_\Lambda^{\circ}\right|\ \dK_X\cap
v^\perp \neq \emptyset\right\}$ which is a dense subset of
$\mathcal{M}_\Lambda^{\circ}$. (Note that the constructed mirror depends on the choice of $\sigma_X$ and $\beta$ for the element $(X,\phi,
\sigma_X, \omega_X, \beta)\in \widetilde{\dM_\Lambda}^j$) 
\begin{prop}\label{existance}
Let $(X,\varphi)\in \mathcal{M}_\Lambda^{\circ}$. Assume that the Kähler cone $\dK_X$ satisfies $\dK_X\cap
v^\perp \neq \emptyset$ 
then
there exists $\sigma_X\in H^{2,0}(X)$, $\omega_X\in \mathcal{K}_X$ and $\beta\in H^2(X,\R)$ such that: 
$$(X,\varphi,\sigma_X,\omega_X,\beta)\in \widetilde{\mathcal{M}_{\Lambda}}^j.$$
\end{prop}
\begin{proof}
Let $(X,\varphi)$ be as in the statement of the proposition. Let $\sigma_X'\in H^{2,0}(X)$. For any $(a+ib)\in
\C^*$, we have: $$\Ima\left[(a+ib)\sigma_X'\right]=a\Ima \sigma'_X+b\Ree \sigma'_X.$$
Choose $(a+ib)\in \C^*$ such that:
$$a(\Ima \sigma'_X\cdot v)+b(\Ree \sigma'_X\cdot v)=0.$$
Then we set $\sigma_X:=(a+ib)\sigma_X'$.
Since by assumption $\dK_X\cap v^\perp \neq \emptyset$,  we can choose $\omega_X\in \mathcal{K}_X$ such that $\omega_X\cdot v=0$.
By rescaling $\omega_X$, we can furthermore assume that $\omega_X^2\in\R\smallsetminus \Q$. Finally, by Lemma \ref{domain}, we can choose $\beta\in H^2(X,\R)$ such that 
$(X,\varphi,\sigma_X,\omega_X,\beta)\in \widetilde{\mathcal{M}_{\Lambda}}^j.$
\end{proof}
Now, we show that our definition of mirror symmetry coincides with the one of Dolgachev \cite{Dolgachev} and
Camere \cite{Camere-2016} for lattice-polarized irreducible symplectic manifolds.
\begin{defi}\label{polarized}
Let $\Lambda$ be a lattice of signature $(3,\rk \Lambda-3)$.
Let $\nu:N\hookrightarrow \Lambda$ be a primitive embedding of a sublattice $N$ of signature $(1,\rk N-1)$.
An \emph{$(N,\nu)$-polarized irreducible symplectic orbifold} is a couple $(X,\varphi)\in\mathcal{M}_{\Lambda}$ such that:
 $$\nu(N)\subset \varphi(\Pic X).$$
 We say that $(X,\varphi)$ is \emph{strictly $(N,\nu)$-polarized} if:
 $$\nu(N)= \varphi(\Pic X).$$
\end{defi}
\begin{prop}\label{polarized}
Let $\nu: N\hookrightarrow \Lambda$ be a primitive embedding, with $N$ of signature $(1,\rk N-1)$, such that $j:U(n)\hookrightarrow \nu(N)^{\bot}$. 
We set $N^{\vee}\coloneqq U(n)^\perp \cap \nu(N)^{\bot}$ and $\nu^{\vee}:N^{\vee}\hookrightarrow \Lambda$ the natural embedding.
Let $(X,\varphi)\in \mathcal{M}_\Lambda^{\circ}$ be a strictly $(N,\nu)$-polarized irreducible symplectic orbifold.  
Then there exists $\sigma_X\in H^{2,0}(X)$, $\omega_X\in \mathcal{K}_X$ and $\beta\in H^2(X,\R)$ such that: 
$$(X,\varphi,\sigma_X,\omega_X,\beta)\in \widetilde{\mathcal{M}_{\Lambda}}^j,$$ 
and $(X^{\vee},\varphi^{\vee})$ is $(N^{\vee},\nu^{\vee})$-polarized, where
$(X^{\vee},\varphi^{\vee},\sigma_{X^{\vee}},\omega_{X^{\vee}},\beta^{\vee}):=m_j(X,\varphi,\sigma_X,\omega_X,\beta)$.
\end{prop}
\begin{proof} 
For simplicity, we denote $\nu(N)$ by $N$. The projectivity criterion (\cite[Theorem 1.2]{Menet-2020}) implies that $X$ is projective, since $\Pic(X)=\phi^{-1}(N)$, has signature $(1,\rk
N -1)$, and therefore the ample cone $\dK_X\cap \phi^{-1}(N_\bR)$
is non-empty.
Choose $\omega_X\in \mathcal{K}_X\cap \varphi^{-1}(N_{\R}
)$ with $\omega_X^2\in\R\smallsetminus \Q$.
Furthermore, fix $\beta\in \varphi^{-1}(N_{\R})$ such that for all $D \in \cW_\Lambda\setminus \varphi^{-1}(N^{\bot}\cap \cW_\Lambda))$:
$$ \pr(\beta)\cdot D\neq 0.$$
In particular, the third condition of Lemma \ref{domain} is satisfied for all $D \in \cW_\Lambda\setminus
\varphi^{-1}(N^{\bot}\cap \cW_\Lambda))$. 

Let $\sigma'_X\in H^{2,0}(X)$ and $(a+ib)\in \C^*$.
Again,
we have: $$\Ima\left[(a+ib)\sigma'_X\right]=a\Ima \sigma'_X+b\Ree \sigma'_X,$$
and we can choose $(a+ib)\in \C^*$ such that:
\begin{equation}
\Ima((a+ib)\sigma'_X)\cdot v=a(\Ima \sigma'_X\cdot v)+b(\Ree \sigma'_X\cdot v)=0.
\label{ab}
\end{equation}
Set $\sigma_X\coloneqq (a+ib)\sigma_X'$.
To use Lemma \ref{domain}, we also need to check the third condition for
 $D\in\varphi^{-1}(N^{\bot}\cap \cW_\Lambda)$. To do this, we will show that 
\begin{equation}
\Ima(\sigma_X)\cdot D=a(\Ima \sigma'_X\cdot D)+b(\Ree \sigma'_X\cdot D)\neq 0,
\label{ab2}
\end{equation}
for all $D\in \varphi^{-1}(\cW_\Lambda\cap N^{\bot})$. Assume, for the sake of contradiction, that \eqref{ab2} does not
hold, i.e. $\Ima(\sigma)\cdot D=0$.
Observe that in this case we have $\sigma_X\cdot (D-cv)=\Rea(\sigma_X)\cdot(D-cv)$,
for any real number $c\in \R$.
Note that $\Rea(\sigma_X) \cdot v=\sigma_X\cdot v \neq 0$, since $v\notin \Pic(X)$.
Hence, by setting $c\coloneqq \frac{\Rea(\sigma_X)\cdot D}{\Rea(\sigma_X)\cdot v}$, one obtains
$\sigma_X\cdot (D-cv)=0$.
This implies that $D-cv\in \Pic X$.
However, we have chosen $D\in \Pic X^{\bot}$ and $v\in \Pic X^{\bot}$. Since the Beauville-Bogomolov form is non-degenerate, we have $D-cv=0$.
This is impossible because $q(D)<0$ and $v^2=0$. 

Therefore, we can apply Lemma \ref{domain} to see that there exists $\lambda\in \R$ such that 
$(X,\varphi,\sigma_X,\omega_X,\lambda\beta)\in \widetilde{\mathcal{M}_{\Lambda}}^j.$
Finally, since by hypothesis $\nu^{\vee}(N^{\vee})^{\bot}=U(n)\oplus \nu(N)$, we obtain from (\ref{equationmiror}) that $\sigma_{X^{\vee}}\in \nu^{\vee}(N_{\C}^{\vee})^{\bot}$,
whence $(X^{\vee},\varphi^{\vee})$ is $(N^{\vee},\nu^{\vee})$-polarized. 
\end{proof}
~\\


\bibliographystyle{alpha}
\bibliography{Literatur}

\noindent
Gr\'egoire \textsc{Menet}

\noindent
Académie militaire de Saint-Cyr Coëtquidan

\noindent
 CReC Saint-Cyr (Centre de recherche de l'académie militaire de Saint-Cyr) 

\noindent 
 56380 Guer, France.

\noindent
{\tt gregoire.menet@ac-amiens.fr}

\bigskip
    
 \noindent
Ulrike \textsc{Rie}\ss

\noindent
Institute for Theoretical Studies - ETH Z\"urich

\noindent 
Clausiusstrasse 47, Building CLV, Z\"urich (Switzerland)

\noindent
{\tt uriess@ethz.ch}

\end{document}